\documentclass{amsart}

\usepackage[utf8]{inputenc}
\usepackage{amsmath}
\usepackage{amssymb}
\usepackage{mathdots}
\usepackage{mathtools}
\usepackage{mathrsfs}
\usepackage{tikz}
\usepackage{tikz-cd}
\usetikzlibrary{decorations.markings,arrows.meta,quotes}
\usepackage[inline]{enumitem}
\usepackage{color}
\usepackage{booktabs}
\definecolor{linkcolor}{rgb}{0,0,0.6}
\usepackage{amsthm}
\usepackage{bbm}
\usepackage{multicol}

\usepackage[english]{babel}
\usepackage{csquotes}

\usepackage[colorlinks=true,
pdfstartview=Fit,
linkcolor= linkcolor,
citecolor= linkcolor,
urlcolor= linkcolor,
hyperindex=true,
hyperfigures=false]
{hyperref}

\usepackage{orcidlink}

\tikzset{
  goodcycle/.style={draw=blue!65!black,fill=blue!8,
    rounded corners=2pt,inner sep=4pt},
  badcycle/.style={draw=red!70!black,fill=red!8,
    rounded corners=2pt,inner sep=4pt}
}

\title[Ekedahl-Oort strata meeting the supersingular locus]{Ekedahl-Oort strata meeting the supersingular locus}
\author{Joseph Muller\,\orcidlink{0000-0002-1546-0910}}\thanks{National Center for Theoretic Sciences, National Taiwan University, No. 1, Sec. 4, Roosevelt Rd., Taipei, Taiwan, muller@ncts.ntu.edu.tw}
\date{}
\newtheorem{theorem}{Theorem}[section]

\newtheorem{proposition}[theorem]{Proposition}
\newtheorem{lemma}[theorem]{Lemma}
\newtheorem{corollary}[theorem]{Corollary}

\theoremstyle{definition}
\newtheorem{definition}[theorem]{Definition}
\theoremstyle{remark}
\newtheorem{remark}[theorem]{Remark}

\begin{document}


\begin{abstract}
We give an explicit combinatorial criterion for an Ekedahl-Oort (EO) stratum in the moduli stack of principally polarized abelian varieties of dimension $g$ in characteristic $p > 2$ to meet the supersingular locus. The criterion is obtained by translating a non-emptiness criterion for basic affine Deligne-Lusztig varieties, into an explicit condition on the elementary sequences indexing the EO strata. Furthermore, we give explicit bounds for the number of EO strata of $p$-rank zero which are disjoint from the supersingular locus. As a result, we prove that asymptotically almost every EO stratum of $p$-rank zero meets the supersingular locus. As an application, we also give a lower bound for the number of EO strata meeting the supersingular locus on unitary PEL Shimura varieties of signature $(a,b)$.
\end{abstract}

\maketitle

\section*{Introduction}

The geometry in positive characteristic of the moduli space of principally polarized abelian varieties is closely related to the invariants of their $p$-divisible groups. Two natural stratifications are defined by the Newton polygon and by the isomorphism class of the polarized $p$-torsion, respectively. The latter is the Ekedahl-Oort (EO) stratification, studied systematically by Oort in \cite{oortStratificationModuliSpace2001a}. Its strata are indexed by the ``elementary sequences'', and their dimensions and closure relations admit explicit combinatorial descriptions. The Newton stratification records the isogeny class of the $p$-divisible group, see \cite{oortNewtonPolygonStrata2001}. An EO stratum need not be contained in a single Newton stratum, and determining which strata meet is a natural problem in the study of these two stratifications.

Fix an odd prime $p$, and write $\mathcal A_g$ for the moduli stack over $\mathbb F_p$ of principally polarized abelian varieties of dimension $g$. In this paper, we consider its supersingular locus $\mathcal S_g$, the closed Newton stratum on which all Newton slopes are $1/2$. We ask which EO strata of $\mathcal A_g$ meet $\mathcal S_g$. Explicit descriptions of these intersections in genera $3$ and $4$ were obtained by Shimada and Takamatsu \cite{shimada_supersingular_2025}, while Groen, Lupoian and Parker \cite{groenIntersectionsEkedahlOortNewton2026} determined the intersections of the EO and Newton strata in genus $5$. General non-emptiness criteria for affine Deligne-Lusztig varieties also apply to this question in arbitrary genus.

Our first goal is to express the answer directly in terms of the elementary sequence indexing an EO stratum. Given such a sequence, we construct a directed graph, the ``cycle graph'', whose vertices are the cycles of an explicit permutation of $\{1, \ldots, g\}$. These cycles are divided into two classes, called good and bad. Theorem \ref{thm:main} states that the EO stratum meets the supersingular locus if and only if every bad cycle is reachable from a good cycle. The construction of the graph is given in Section \ref{Section1}. Equivalently, it characterizes the polarized $\mathrm{BT}_1$-modules which occur as the $p$-torsion of principally quasi-polarized supersingular $p$-divisible groups.

The proof uses the group-theoretic description of the EO stratification through the hyperspecial EKOR parametrization of He and Rapoport \cite{heStratificationsReductionShimura2017}. We identify its indices explicitly with elementary sequences, and reduce the question to the non-emptiness of a basic affine Deligne-Lusztig variety. We then apply Theorem \ref{thm:schremmer}, proved by Schremmer in \cite{schremmerNewtonStrataLevi2024} Proposition 5, following a conjecture of Lim in \cite{lim_nonemptiness_2026}. The main step is to translate the group theoretic conditions in that theorem into the reachability condition on the cycle graph. These comparisons and computations occupy Sections \ref{Section2} and \ref{Section3}.

We then study the number $N_g$ of EO strata meeting $\mathcal S_g$. Among the $2^g$ EO strata of $\mathcal A_g$, precisely $2^{g-1}$ have $p$-rank zero, and only these can meet the supersingular locus. Theorem \ref{thm:asymptotic-count} gives explicit upper and lower bounds for the number of $p$-rank zero EO strata which do not meet $\mathcal S_g$, valid for every $g \geq 3$. These imply 
\begin{equation*}
    2^{g-1} - N_g = \Theta \left(\frac{2^g}{g}\right), 
\end{equation*}
and therefore $N_g \sim 2^{g-1}$ as $g$ goes to infinity. Thus, asymptotically almost every EO stratum of $p$-rank zero meets the supersingular locus, and the proportion of EO strata of $p$-rank zero which do not meet $\mathcal S_g$ has order $1/g$. The proof in Section \ref{sec:enumeration} uses a particular bijection between elementary sequences of $p$-rank zero and compositions of the integer $g$, adapted to the construction of the cycle graph. It allows us to compare the number of $p$-rank zero EO strata disjoint from $\mathcal S_g$ with the number of compositions whose first part is maximal. 

Finally, in Section \ref{sec:unitary-application}, we record an application of these results to unitary EO strata, using the forgetful and Serre tensor constructions considered in \cite{andrews_classification_2026} and \cite{andrewsEkedahlOortStrata$mathsfGUq22$2025a}. For a unitary PEL Shimura variety of signature $(a,b)$ with $a \geq b \geq 1$, at hyperspecial level at an odd prime inert in the underlying imaginary quadratic field, Corollary \ref{cor:unitary-arbitrary} shows that at least $N_b$ distinct EO strata meet the supersingular locus.

Independently and simultaneously, Shimada obtained in \cite{shimada_number_2026} a criterion for EO strata to meet the supersingular locus, thus giving a statement equivalent to Theorem \ref{thm:main} below, and proved that $N_g \sim 2^{g-1}$ as well. Neither author was aware of the other's work during its development.

\section{Ekedahl-Oort stratification and cycle graphs}\label{Section1}

Throughout, $p$ will always refer to an odd prime number. Over any perfect field of characteristic $p$, $\sigma$ denotes the automorphism $x \mapsto x^p$. We write $\mathbb D(\cdot)$ for the \textit{covariant} Dieudonné functor. We write $\mathcal A_g$ for the moduli stack over $\mathbb F_p$ of principally polarized abelian varieties of dimension $g \geq 1$, and $\mathcal S_g \subseteq \mathcal A_g$ for its supersingular locus, see \cite{oortStratificationModuliSpace2001a,oortNewtonPolygonStrata2001}.

\subsection{The Ekedahl-Oort stratification}

An \textit{elementary sequence} is a sequence of integers $0 = \varphi_0 \leq \varphi_1 \leq \cdots \leq \varphi_g$ such that $\varphi_i \leq \varphi_{i-1}+1$ for all $1\leq i \leq g$. We define
\begin{equation*}
    \forall 1 \leq i \leq g, \quad\delta_i \coloneq \varphi_i - \varphi_{i-1} \in \{0,1\}, \qquad P_{\delta} \coloneq \{i \mid \delta_i = 1\}, \qquad Z_{\delta} \coloneq \{i \mid \delta_i = 0\}.
\end{equation*}
We call $\delta = (\delta_i)_{1\leq i \leq g} \in \{0,1\}^g$ the \textit{increment sequence} of $\varphi$. When the context is clear, we will write $P = P_{\delta}$ and $Z = Z_{\delta}$. The elementary sequence $\varphi$ is determined by $\delta$ via 
\begin{equation*}
    \varphi_i = \sum_{j=1}^i \delta_j,
\end{equation*}
so that there is a total of $2^g$ elementary sequences. We define a permutation $\pi_{\delta} \in \mathfrak S_g$ as follows 
\begin{equation}\label{eq:pi-definition}
    \pi_{\delta}(i) \coloneq \begin{cases}
        \varphi_i & \text{if } i \in P,\\
        g - (i - \varphi_i) +1 & \text{if } i \in Z. \end{cases}
\end{equation}
In other words, $\pi_{\delta}$ re-orders the integers in $P$ increasingly from $1$ to $\varphi_g = |P|$, and the integers in $Z$ decreasingly from $g$ to $\varphi_g + 1 = |P|+1$. 

\begin{definition}
    A \textit{$\mathrm{BT}_1$-module} over a perfect field $k$ of characteristic $p$ is a finite-dimensional $k$-vector space $M$ equipped with a $\sigma$-linear operator $F:M \to M$ and a $\sigma^{-1}$-linear operator $V: M \to M$ such that 
    \begin{equation*}
        FV = VF = 0, \qquad \mathrm{Ker}(F) = \mathrm{Im}(V), \qquad \mathrm{Ker}(V) = \mathrm{Im}(F).
    \end{equation*}
    A \textit{polarized $\mathrm{BT}_1$-module} over $k$ is a $\mathrm{BT}_1$-module $M$ equipped with a non-degenerate alternating $k$-bilinear form $\langle \cdot,\cdot \rangle$ such that 
    \begin{equation*}
        \forall x,y \in M, \qquad \langle Fx, y \rangle = \langle x, Vy \rangle^p.
    \end{equation*}
\end{definition}

If $M$ is a polarized $\mathrm{BT}_1$-module, then $\dim_k M$ is even. Moreover, if $(A,\lambda) \in \mathcal A_g(k)$ then the pair $M_{A,\lambda} = (\mathbb D(A[p]), \langle \cdot, \cdot \rangle_{\lambda})$ is a polarized $\mathrm{BT}_1$-module, where $\langle \cdot, \cdot \rangle_{\lambda}$ is induced by the principal polarization $\lambda$. Over an algebraically closed field $k$, the isomorphism classes of polarized $\mathrm{BT}_1$-modules of dimension $2g$ are in bijection with elementary sequences $\varphi$ as above, see \cite{oortStratificationModuliSpace2001a} Section 9 and also \cite{priesBT_1GroupSchemes2021} Section 4 for a modern account. Oort provides the construction of a standard polarized $\mathrm{BT}_1$-module $M_{\delta}$ over $\mathbb F_p$, such that its scalar extension to $k$ is a representative of the isomorphism class corresponding to the elementary sequence $\varphi$ with increment sequence $\delta$. His construction can be restated as follows, namely $M_{\delta}$ (denoted $A_{\varphi}$ in \cite{oortStratificationModuliSpace2001a}) has a symplectic basis 
\begin{equation}\label{eq:symplectic-basis}
    X_1, \ldots, X_g, Y_1, \ldots, Y_g, \qquad \langle X_i, Y_j \rangle = \delta_{ij}, \qquad \langle X_i, X_j \rangle = \langle Y_i, Y_j \rangle = 0,
\end{equation}
where $1 \leq i,j \leq g$ and $\delta_{ij}$ is the Kronecker delta, and such that 
\begin{align}
    \forall i \in P, & \qquad F X_i = X_{\pi_{\delta}(i)}, & V Y_{\pi_{\delta}(i)} & = Y_i, & V X_{\pi_{\delta}(i)} & = 0, \label{eq:normal-P}\\
    \forall i \in Z, & \qquad F X_i = Y_{\pi_{\delta}(i)}, & V X_{\pi_{\delta}(i)} & = -Y_i, & V Y_{\pi_{\delta}(i)} & = 0,\label{eq:normal-Z}
\end{align}
and $F Y_{i} = 0$ for all $1 \leq i \leq g$. 

\begin{definition}
    Given an elementary sequence $\varphi$ with increment sequence $\delta$, the associated \textit{Ekedahl-Oort (EO) stratum} $\mathcal A_{g,\delta}$ is the reduced locally closed substack of $\mathcal A_g$ whose geometric points, say over an algebraically closed field $k$, are those principally polarized abelian varieties $(A, \lambda)$ such that $M_{A,\lambda} \simeq (M_{\delta}, \langle \cdot, \cdot \rangle) \otimes_{\mathbb F_p} k$ as polarized $\mathrm{BT}_1$-modules.
\end{definition}

All together, the EO strata form the Ekedahl-Oort stratification of $\mathcal A_g$. Each stratum $\mathcal A_{g,\delta}$ is a non-empty smooth Deligne-Mumford stack over $\mathbb F_p$, of pure dimension
\begin{equation}\label{eq:EO-dimension}
    \dim \mathcal A_{g,\delta} = \sum_{i=1}^g \varphi_i = \sum_{i=1}^g (g+1-i) \delta_i,
\end{equation}
see \cite{oortStratificationModuliSpace2001a} Theorem 1.2 for non-emptiness and dimension, and \cite{viehmann_ekedahloort_2013} Proposition 10.3 for smoothness. For convenience, we recall that if $(A,\lambda) \in \mathcal A_{g,\delta}(k)$, then 
\begin{equation}\label{eq:classical-invariants}
    a(A) = g - \varphi_g = |Z|, \qquad 
    p\mathrm{-rk}(A) = \max\{ 0 \leq i \leq g \mid \varphi_i = i\} =  \begin{cases}
        \min(Z) - 1 & \text{if } Z \not = \emptyset, \\
        g & \text{if } Z = \emptyset,
    \end{cases} 
\end{equation}
where $a(A)$ is the $a$-number of $A$ and $p\mathrm{-rk}(A)$ is its $p$-rank.

\subsection{The cycle graph of an elementary sequence}

Let us fix an elementary sequence $\varphi$ and its increment sequence $\delta$. We define 
\begin{equation}\label{eq:signs}
    \varepsilon_i \coloneq 2\delta_i - 1 = \begin{cases}
        +1 & \text{if } \delta_i = 1,\\
        -1 & \text{if } \delta_i = 0.
    \end{cases}
\end{equation}
We consider the decomposition of $\pi_{\delta} \in \mathfrak S_g$ as a product of cycles. Here, cycles are allowed to have cardinality $1$, thus corresponding to fixed points of $\pi_{\delta}$. Let $C$ be a cycle of $\pi_{\delta}$ and put
\begin{equation*}
    \varepsilon(C) \coloneq \prod_{i\in C} \varepsilon_i = (-1)^{|C\cap Z|}.
\end{equation*}
If $\varepsilon(C) = 1$, an \textit{orientation} of $C$ is a map
\begin{equation}\label{eq:orientation}
    t_C : C \longrightarrow \{ \pm 1 \} \text{ such that } \forall i \in C, \quad t_C(\pi_{\delta}(i)) = \varepsilon_i t_C(i).
\end{equation}
There are exactly two orientations, which differ by multiplication by $-1$. 

\begin{definition}\label{def:bad-cycle}
    A cycle $C$ of $\pi_{\delta}$ is said to be \textit{bad} if the following three conditions are satisfied:
    \begin{itemize}
        \item $C \cap P \not = \emptyset$, 
        \item $\varepsilon(C) = 1$, 
        \item there is an orientation $t_C$ such that $t_C(i) = 1$ for all $i \in C \cap P$.
    \end{itemize}
    If $C$ is bad, the orientation $t_C$ given by the third condition is unique, and will always be denoted by $t_C$. Every cycle which is not bad is called \textit{good}.
\end{definition}

Given a bad cycle $C$, starting with some element of $C \cap P$, one may write the cycle schematically as 
\begin{equation*}
    P Z^{a_1} P Z^{a_2} \cdots P Z^{a_r},
\end{equation*}
for some $r \geq 1$ and $a_1, \ldots, a_r \geq 0$. In other words, the integers $a_i$ are the numbers of elements of $C \cap Z$ lying between two consecutive elements of $C \cap P$ in the cyclic order. Since $\varepsilon(C) = 1$, we know that the sum $a_1 + \cdots + a_r$ is even. Then the existence of an orientation $t_C$ such that $t_C(i) = 1$ for all $i \in C \cap P$ is equivalent to the requirement that $a_i$ is even for all $1 \leq i \leq r$. 

\begin{definition}\label{defi:cycle_graph}
    We write $\mathcal B_{\delta}$ and $\mathcal G_{\delta}$ respectively for the set of bad cycles and good cycles of $\pi_{\delta}$. The \textit{cycle graph} $\Gamma_{\delta}$ of $\pi_{\delta}$ is the directed graph whose vertices are the cycles $C \in \mathcal B_{\delta} \sqcup \mathcal G_{\delta}$, and such that there is an arrow $C \to C'$ if and only if the following conditions hold:
    \begin{itemize}
        \item $C \not = C'$,
        \item $C' \in \mathcal B_{\delta}$,
        \item there is $i \in C \cap P$ and $j \in C' \cap Z$ such that $i < j$ and $t_{C'}(j) = 1$.
    \end{itemize}
\end{definition}

In particular, the head of any arrow in $\Gamma_{\delta}$ is bad, whereas the tail can be either good or bad.

\begin{definition}\label{def:star}
    We say that $\delta$ satisfies condition $(*)_{\delta}$ if every bad cycle is reachable in $\Gamma_{\delta}$ from a good cycle.
\end{definition}

In other words, condition $(*)_{\delta}$ holds if and only if for every $C \in \mathcal B_{\delta}$, there exists $r \geq 1$ and cycles $C_0, \ldots , C_r = C$ such that $C_0 \in \mathcal G_{\delta}$ and $C_{i-1} \to C_{i}$ for every $1 \leq i \leq r$.

\begin{lemma}\label{lem:no-incoming}
    Condition $(*)_{\delta}$ fails if and only if there is a non-empty subset $\mathcal U \subseteq \mathcal B_{\delta}$ such that for every arrow $C \to C'$, we have
    \begin{equation*}
        C' \in \mathcal U \quad \implies \quad C \in \mathcal U.
    \end{equation*}
\end{lemma}

\begin{proof}
If $(*)_{\delta}$ fails, take for $\mathcal U$ the set of all bad cycles which are not reachable from a good cycle. Any arrow into $\mathcal U$ from its complement would make its target reachable from a good cycle, leading to a contradiction. Conversely, if $\mathcal U$ is as in the Lemma, assume that a sequence of cycles $C_0, \ldots , C_r$ with $r \geq 1$, $C_0 \in \mathcal G_{\delta}$, $C_r \in \mathcal U$ and $C_{i-1} \to C_i$ for all $1 \leq i \leq r$ exists. Let $0 \leq j \leq r$ be the smallest integer such that $C_j \in \mathcal U$. Then necessarily $j \geq 1$ and the arrow from $C_{j-1}$ to $C_j$ would imply that $C_{j-1} \in \mathcal U$, contradicting the minimality of $j$.
\end{proof}

Our main result is the following.

\begin{theorem}\label{thm:main}
    Let $\varphi$ be an elementary sequence and $\delta$ be its increment sequence. Let $k$ be an algebraically closed field of characteristic $p$. The following conditions are equivalent.
    \begin{enumerate}[label = \textup{(\roman*)}]
        \item there is a principally quasi-polarized supersingular $p$-divisible group $X$ over $k$ such that $\mathbb D(X[p]) \simeq (M_{\delta}, \langle \cdot, \cdot \rangle) \otimes_{\mathbb F_p} k$ as polarized $\mathrm{BT}_1$-modules,
        \item $\mathcal A_{g,\delta} \cap \mathcal S_g \not = \emptyset$,
        \item condition $(*)_{\delta}$ holds.
\end{enumerate}
\end{theorem}

The implication \textup{(ii)} $\implies$ \textup{(i)} is due to $\mathcal A_{g,\delta} \subset \mathcal A_g$ being of finite type, and the converse \textup{(i)} $\implies$ \textup{(ii)} follows from Proposition 5.3 of \cite{harashita_generic_2010}, stating that any principally quasi-polarized $p$-divisible group over $k$ is isomorphic to the $p$-divisible group of a principally polarized abelian variety over $k$. Thus, the new contribution is the equivalence with \textup{(iii)}. 

\begin{lemma}\label{lem:positive-p-rank}
    If $\delta_1 = 1$, then none of the three conditions in Theorem \ref{thm:main} holds.
\end{lemma}

\begin{proof}
A supersingular abelian variety has $p$-rank zero, so that $\mathcal A_{g,\delta} \cap \mathcal S_g \not = \emptyset$ implies $\delta_1 = 0$ by \eqref{eq:classical-invariants}. On the other hand, if $\delta_1 = 1$ then $\pi_{\delta}(1) = 1$. The one-element cycle $\{1\}$ is bad and clearly unreachable in $\Gamma_{\delta}$, implying that $(*)_{\delta}$ fails. 
\end{proof}

\subsection{Examples}

Let us illustrate condition $(*)_{\delta}$ and the use of Theorem \ref{thm:main} in some examples. In the figures below, blue vertices are good cycles and red vertices are bad cycles.

\begin{figure}[htbp]
\centering
\begin{minipage}[t]{0.48\textwidth}
\centering
\textup{(a)} $g=3$, $\delta = \mathtt{001}$\par $\varphi=(0,0,1)$\par
\smallskip
\begin{tikzpicture}
  \node[goodcycle] (G1) at (0,0) {$(1,3)$};
  \node[goodcycle] (G2) at (2.7,0) {$(2)$};
\end{tikzpicture}\par
\smallskip
No bad cycle. $(*)_{\delta}$ holds.
\end{minipage}
\hfill
\begin{minipage}[t]{0.48\textwidth}
\centering
\textup{(b)} $g=3$, $\delta=\mathtt{010}$\par
$\varphi=(0,1,1)$\par
\smallskip
\begin{tikzpicture}
  \node[badcycle] (B) at (0,0) {$(1,3,2)$};
\end{tikzpicture}\par
\smallskip
No good cycle. $(*)_{\delta}$ fails.
\end{minipage}

\vspace{5mm}

\begin{minipage}[t]{0.48\textwidth}
\centering
\textup{(c)} $g=4$, $\delta=\mathtt{0100}$\par
$\varphi=(0,1,1,1)$\par
\smallskip
\begin{tikzpicture}
  \node[goodcycle] (G) at (0,0) {$(3)$};
  \node[badcycle] (B) at (2.8,0) {$(1,4,2)$};
\end{tikzpicture}\par
\smallskip
Good and bad cycles, no arrow. \\
$(*)_{\delta}$ fails.
\end{minipage}
\hfill
\begin{minipage}[t]{0.48\textwidth}
\centering
\textup{(d)} $g=7$, $\delta=\mathtt{0101101}$\par
$\varphi=(0,1,1,2,3,3,4)$\par
\smallskip
\begin{tikzpicture}
  \node[goodcycle] (G) at (0,0) {$(1,7,4,2)$};
  \node[badcycle] (B) at (3.5,0) {$(3,6,5)$};
  \draw[->] (G) -- (B);
\end{tikzpicture}\par
\smallskip
Good and bad cycles with arrow. \\
$(*)_{\delta}$ holds.
\end{minipage}
\caption{Four basic examples of cycle graphs $\Gamma_{\delta}$.}
\label{fig:basic-cycle-graphs}
\end{figure}
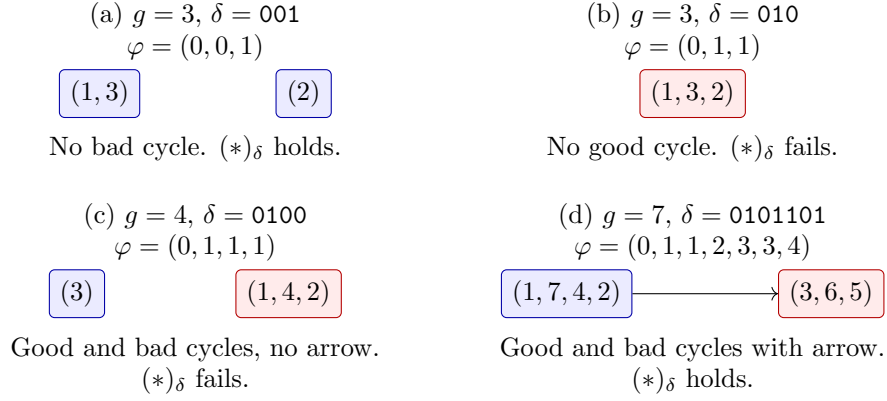

\begin{figure}[htbp]
\centering
\textup{(e)} $g=10$, $\delta=\mathtt{0100111001}$\par
$\varphi=(0,1,1,1,2,3,4,4,4,5)$\par
\smallskip
\begin{tikzpicture}
  \node[goodcycle] (G) at (0,0) {$(1,10,5,2)$};
  \node[badcycle] (B1) at (4.2,0.65) {$(3,9,6)$};
  \node[badcycle] (B2) at (4.2,-0.65) {$(4,8,7)$};
  \draw[->] (G) -- (B1);
  \draw[->] (G) -- (B2);
\end{tikzpicture}

\vspace{5mm}

\textup{(f)} $g=11$, $\delta=\mathtt{00011010001}$\par
$\varphi=(0,0,0,1,2,2,3,3,3,3,4)$\par
\smallskip
\begin{tikzpicture}
  \node[goodcycle] (G) at (0,0.65) {$(1,11,4)$};
  \node[badcycle] (B1) at (0,-0.65) {$(2,10,5)$};
  \node[badcycle] (B2) at (5.0,0) {$(3,9,6,8,7)$};
  \draw[->] (G) -- (B2);
  \draw[->] (B1) -- (B2);
\end{tikzpicture}

\vspace{5mm}

\textup{(g)} $g=13$, $\delta=\mathtt{0100110110010}$\par
$\varphi=(0,1,1,1,2,3,3,4,5,5,5,6,6)$\par
\smallskip
\begin{tikzpicture}
  \node[goodcycle] (G) at (0,0) {$(3,12,6)$};
  \node[badcycle] (B1) at (4.0,0) {$(1,13,7,10,9,5,2)$};
  \node[badcycle] (B2) at (8.2,0) {$(4,11,8)$};
  \draw[->] (G) -- (B1);
  \draw[->] (B1) -- (B2);
\end{tikzpicture}
\caption{Three more examples of cycle graphs $\Gamma_{\delta}$.}
\label{fig:larger-cycle-graphs}
\end{figure}
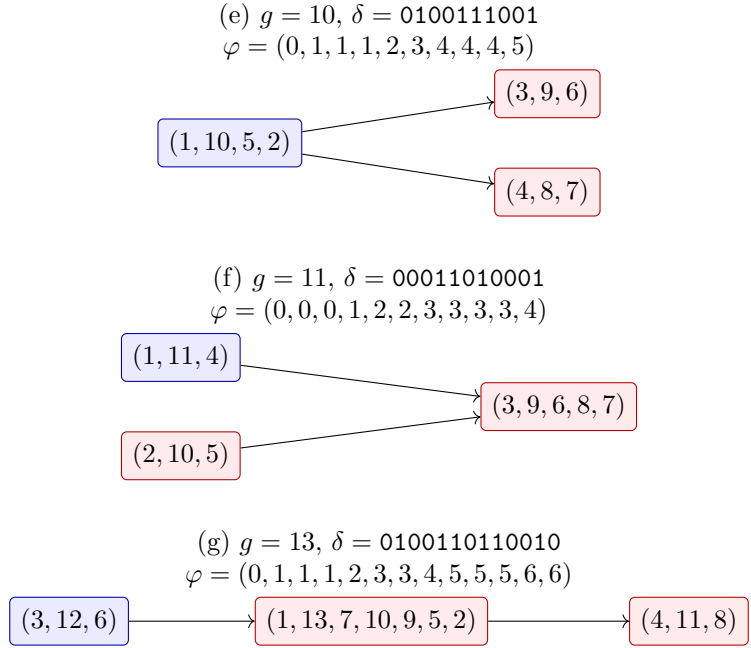

We give explanations for the examples (a)-(g) of Figures \ref{fig:basic-cycle-graphs} and \ref{fig:larger-cycle-graphs}.
\begin{enumerate}[label=\textup{(\alph*)}]
\item Here $P = \{3\}$ and $\pi_{\delta}= (1,3)(2)$. The cycle $(1,3)$ is good because it contains an odd number of elements of $Z$, while the cycle $(2)$ is good because it does not meet $P$. Thus $\mathcal B_{\delta} = \emptyset$, and condition $(*)_{\delta}$ holds vacuously.
\item The permutation $\pi_{\delta} = (1,3,2)$ consists of the single bad cycle $(1,3,2)$. Since there is no good cycle, condition $(*)_{\delta}$ fails.
\item We have $\pi_{\delta} = (1,4,2)(3)$, with $(1,4,2)$ being bad and $(3)$ being good. Since the good cycle does not meet $P$, there is no arrow starting from it. Hence the bad cycle is unreachable and condition $(*)_{\delta}$ fails.
\item We have $\pi_{\delta} = (3,6,5)(1,7,4,2)$, with $(3,6,5)$ being bad and $(1,7,4,2)$ being good. The good cycle meets $P$ at $\{2,4,7\}$, and the orientation $t_C$ of the bad cycle satisfies $t_C(5) = 1$, $t_C(3) = 1$ and $t_C(6) = -1$. Since $3 \in Z$, the inequality $2 < 3$ gives the arrow $(1,7,4,2) \to (3,6,5)$, so condition $(*)_{\delta}$ holds.
\item We have $\pi_{\delta} = (1,10,5,2)(3,9,6)(4,8,7)$, the first cycle being good and the two others being bad. The integer $2$ belongs to $P$, whereas the integers $3$ and $4$ belong to $Z$ and the orientations take value $+1$ at them. The inequalities $2 < 3$ and $2 < 4$ give the two arrows with tail the good cycle $(1,10,5,2)$. Besides, there is no arrow between the two bad cycles. The condition $(*)_{\delta}$ holds.
\item We have $\pi_{\delta} = (1,11,4)(2,10,5)(3,9,6,8,7)$, the first cycle being good and the two others being bad. The integers $4$ and $5$ belong to $P$, and the orientation on the longest bad cycle satisfies $t_C(7) = 1$, $t_C(3) = 1$, $t_C(9) = -1$, $t_C(6) = 1$, $t_C(8) = -1$. The inequalities $4 < 6$ and $5 < 6$ give the two displayed arrows with head the bad cycle $(3,9,6,8,7)$. Besides, the orientation on the shortest bad cycle satisfies $t_C(5) = 1$, $t_C(2) = 1$ and $t_C(10) = -1$. Since $2$ is smaller than any element of $(3,9,6,8,7)$, there is no other arrow. The bad cycle $(2,10,5)$ being unreachable, condition $(*)_{\delta}$ fails.
\item We have $\pi_{\delta} = (3,12,6)(1,13,7,10,9,5,2)(4,11,8)$, the first cycle being good and the two others being bad. The integers $6$ and $12$ belong to $P$, as well as $2$, $5$ and $9$. The orientations on the bad cycles satisfy $t_C(5) = 1$, $t_C(2) = 1$, $t_C(1) = 1$, $t_C(13) = -1$, $t_C(7) = 1$, $t_C(10) = -1$, $t_C(9) = 1$ on one hand, and $t_C(8) = 1$, $t_C(4) = 1$, $t_C(11) = -1$ on the other hand. The inequality $6 < 7$ gives the arrow $(3,12,6) \to (1,13,7,10,9,5,2)$, and the inequality $2 < 4$ gives the arrow $(1,13,7,10,9,5,2) \to (4,11,8)$. There is no other arrow, and every bad cycle can be reached from the good cycle. Therefore condition $(*)_{\delta}$ holds.
\end{enumerate}

By Theorem \ref{thm:main}, the EO strata in (a), (d), (e) and (g) meet the supersingular locus, whereas those in (b), (c) and (f) do not.

Let us write 
\begin{equation*}
    N_g \coloneq \# \left\{ \delta \in \{0,1\}^g \mid (*)_{\delta} \text{ holds} \right\}.
\end{equation*}
By Theorem \ref{thm:main}, this is the number of EO strata in $\mathcal A_g$ which meet $\mathcal S_g$. Using a computer program, directly checking condition $(*)_{\delta}$ for the $2^{g-1}$ increment sequences with $p$-rank zero for $1 \leq g \leq 20$, yields the values recorded in Table \ref{tab:small-genus-counts}. These values agree with those independently found by Shimada in \cite{shimada_number_2026} Table 1, which even contains the values of $N_g$ up to $g = 38$. For completeness, the increment sequences contributing to $N_7$ and $N_8$ are listed in Tables \ref{tab:g7-increment-words} and \ref{tab:g8-increment-words}.

\begin{table}[htbp]
    \centering
    \setlength{\tabcolsep}{4pt}
    \begin{tabular}{c|rrrrrrrrrr}
        $g$ & 1 & 2 & 3 & 4 & 5 & 6 & 7 & 8 & 9 & 10 \\
        \hline
        $N_g$ & $1$ & $2$ & $3$ & $6$ & $11$ & $22$ & $44$ & $89$ & $181$ & $370$ \\
        \noalign{\bigskip}
        $g$ & 11 & 12 & 13 & 14 & 15 & 16 & 17 & 18 & 19 & 20 \\
        \hline
        $N_g$ & $753$ & $1\,538$ & $3\,137$ & $6\,381$ & $12\,981$ & $26\,360$ & $53\,454$ & $108\,291$ & $219\,106$ & $442\,858$ \\
        \noalign{\bigskip}
    \end{tabular}
    \caption{The number of EO strata meeting the supersingular locus for $1 \leq g \leq 20$.}
    \label{tab:small-genus-counts}
\end{table}

\begin{table}[htbp]
\centering
\begingroup
\footnotesize
\setlength{\tabcolsep}{5pt}
\begin{tabular}{@{}llllll@{}}
$\mathtt{0000000}$ & $\mathtt{0000001}$ & $\mathtt{0000010}$ & $\mathtt{0000011}$ & $\mathtt{0000100}$ & $\mathtt{0000101}$ \tabularnewline
$\mathtt{0000110}$ & $\mathtt{0000111}$ & $\mathtt{0001001}$ & $\mathtt{0001011}$ & $\mathtt{0001100}$ & $\mathtt{0001101}$ \tabularnewline
$\mathtt{0001111}$ & $\mathtt{0010001}$ & $\mathtt{0010011}$ & $\mathtt{0010101}$ & $\mathtt{0010110}$ & $\mathtt{0010111}$ \tabularnewline
$\mathtt{0011010}$ & $\mathtt{0011011}$ & $\mathtt{0011101}$ & $\mathtt{0011111}$ & $\mathtt{0100001}$ & $\mathtt{0100010}$ \tabularnewline
$\mathtt{0100011}$ & $\mathtt{0100100}$ & $\mathtt{0100101}$ & $\mathtt{0100110}$ & $\mathtt{0100111}$ & $\mathtt{0101001}$ \tabularnewline
$\mathtt{0101011}$ & $\mathtt{0101101}$ & $\mathtt{0101110}$ & $\mathtt{0101111}$ & $\mathtt{0110001}$ & $\mathtt{0110010}$ \tabularnewline
$\mathtt{0110011}$ & $\mathtt{0110101}$ & $\mathtt{0110111}$ & $\mathtt{0111001}$ & $\mathtt{0111010}$ & $\mathtt{0111011}$ \tabularnewline
$\mathtt{0111101}$ & $\mathtt{0111111}$ & & & &
\end{tabular}
\endgroup
\caption{The $44$ increment sequences in genus $7$ for which
condition $(*)_{\delta}$ holds.}
\label{tab:g7-increment-words}
\end{table}

\begin{table}[htbp]
\centering
\begingroup
\footnotesize
\setlength{\tabcolsep}{5pt}
\begin{tabular}{@{}llllll@{}}
$\mathtt{00000000}$ & $\mathtt{00000001}$ & $\mathtt{00000010}$ & $\mathtt{00000011}$ & $\mathtt{00000100}$ & $\mathtt{00000101}$ \tabularnewline
$\mathtt{00000110}$ & $\mathtt{00000111}$ & $\mathtt{00001000}$ & $\mathtt{00001001}$ & $\mathtt{00001010}$ & $\mathtt{00001011}$ \tabularnewline
$\mathtt{00001100}$ & $\mathtt{00001101}$ & $\mathtt{00001110}$ & $\mathtt{00001111}$ & $\mathtt{00010001}$ & $\mathtt{00010011}$ \tabularnewline
$\mathtt{00010100}$ & $\mathtt{00010101}$ & $\mathtt{00010111}$ & $\mathtt{00011001}$ & $\mathtt{00011011}$ & $\mathtt{00011110}$ \tabularnewline
$\mathtt{00011111}$ & $\mathtt{00100001}$ & $\mathtt{00100011}$ & $\mathtt{00100101}$ & $\mathtt{00100110}$ & $\mathtt{00100111}$ \tabularnewline
$\mathtt{00101001}$ & $\mathtt{00101010}$ & $\mathtt{00101011}$ & $\mathtt{00101110}$ & $\mathtt{00101111}$ & $\mathtt{00110010}$ \tabularnewline
$\mathtt{00110011}$ & $\mathtt{00110100}$ & $\mathtt{00110101}$ & $\mathtt{00110110}$ & $\mathtt{00110111}$ & $\mathtt{00111001}$ \tabularnewline
$\mathtt{00111011}$ & $\mathtt{00111110}$ & $\mathtt{00111111}$ & $\mathtt{01000001}$ & $\mathtt{01000010}$ & $\mathtt{01000011}$ \tabularnewline
$\mathtt{01000100}$ & $\mathtt{01000101}$ & $\mathtt{01000110}$ & $\mathtt{01000111}$ & $\mathtt{01001001}$ & $\mathtt{01001011}$ \tabularnewline
$\mathtt{01001100}$ & $\mathtt{01001101}$ & $\mathtt{01001111}$ & $\mathtt{01010001}$ & $\mathtt{01010011}$ & $\mathtt{01010101}$ \tabularnewline
$\mathtt{01010110}$ & $\mathtt{01010111}$ & $\mathtt{01011001}$ & $\mathtt{01011010}$ & $\mathtt{01011011}$ & $\mathtt{01011101}$ \tabularnewline
$\mathtt{01011111}$ & $\mathtt{01100001}$ & $\mathtt{01100010}$ & $\mathtt{01100011}$ & $\mathtt{01100100}$ & $\mathtt{01100101}$ \tabularnewline
$\mathtt{01100110}$ & $\mathtt{01100111}$ & $\mathtt{01101001}$ & $\mathtt{01101011}$ & $\mathtt{01101101}$ & $\mathtt{01101110}$ \tabularnewline
$\mathtt{01101111}$ & $\mathtt{01110001}$ & $\mathtt{01110010}$ & $\mathtt{01110011}$ & $\mathtt{01110101}$ & $\mathtt{01110111}$ \tabularnewline
$\mathtt{01111001}$ & $\mathtt{01111010}$ & $\mathtt{01111011}$ & $\mathtt{01111101}$ & $\mathtt{01111111}$ &
\end{tabular}
\endgroup
\caption{The $89$ increment sequences in genus $8$ for which
condition $(*)_{\delta}$ holds.}
\label{tab:g8-increment-words}
\end{table}

\section{Affine Deligne-Lusztig varieties}\label{Section2}

In this section, we introduce the notions that will be used for the proof of Theorem \ref{thm:main}.

\subsection{The group theoretic setting}\label{subsec:local-siegel-datum}

We fix a free $\mathbb Z_p$-module $\Lambda_0$ of rank $2g$ with basis ordered as $u_1, \ldots , u_g, u_{-1}, \ldots , u_{-g}$. Unless mentioned otherwise, matrices of endomorphisms of $\Lambda_0$ and scalar extensions alike will be considered with respect to this ordered basis. We equip $\Lambda_0$ with the perfect alternating form $\psi$ determined by

\begin{equation*}
    \forall 1 \leq |i|, |j| \leq g, \qquad \psi(u_i,u_j) = \mathrm{sgn}(i)\delta_{i,-j}.  
\end{equation*}
We write $H \coloneq \Lambda_0 \otimes_{\mathbb Z_p} \mathbb Q_p$, equipped with the scalar extension of $\psi$. Let $\mathcal G = \mathrm{GSp}(\Lambda_0, \psi)$ be the reductive group scheme over $\mathbb Z_p$ defined, for every $\mathbb Z_p$-algebra $R$, by
\begin{equation*}
    \mathcal G(R) = \left\{ h \in \mathrm{GL}(\Lambda_0 \otimes_{\mathbb Z_p} R) \;\middle | \;
    \begin{array}{c}
        \text{there exists } c(h) \in R^{\times} \text{ such that}\\ \psi(hx, hy) = c(h) \psi(x,y) \text{ for all } x,y 
    \end{array} \right\}.
\end{equation*}
We write
\begin{equation*}
    G = \mathcal G_{\mathbb Q_p} = \mathrm{GSp}(H,\psi), \qquad c: G \longrightarrow \mathbb G_m,
\end{equation*}
for its generic fibre and its similitude character. Thus $G^{\mathrm{der}} = \mathrm{Sp}(H, \psi)$ and the root system of the derived group is of type $C_g$. Let us write $\mathcal O_L \coloneq W(\overline{\mathbb F_p})$ and $L \coloneq \mathrm{Frac}(\mathcal O_{L})$, so that $L$ is the completion of the maximal unramified extension of $\mathbb Q_p$. We still denote by $\sigma$ the lift of $x \mapsto x^p$ on $\mathcal O_L$ and on $L$. The standard lattice $\Lambda = \Lambda_0 \otimes_{\mathbb Z_p} \mathcal O_L$ is self-dual, and $K \coloneq \mathcal G(\mathcal O_L) \subset G(L)$ is the corresponding hyperspecial subgroup. Let $\mathcal T \subset \mathcal G$ be the diagonal maximal torus and put $T = \mathcal T_{\mathbb Q_p}$. Explicitly,
\begin{equation}\label{eq:diagonal-torus}
    \mathcal T(R) = \left\{ \mathrm{diag}(a_1, \ldots, a_g, ca_1^{-1}, \ldots, ca_g^{-1}) \mid a_1, \ldots, a_g, c\in R^{\times} \right\}.
\end{equation}
For $1 \leq i \leq g$, put
\begin{equation*}
    \mathcal F_i = \bigoplus_{j=1}^{i} \mathbb Z_p u_j, \qquad \mathcal F_i^- = \bigoplus_{j=1}^{i} \mathbb Z_p u_{-j}.
\end{equation*}
Let $\mathcal B, \mathcal B^{-} \subseteq \mathcal G$ be respectively the stabilizers of the two complete isotropic flags
\begin{equation*}
    \{0\} \subset \mathcal F_1 \subset \cdots \subset \mathcal F_g, \qquad \{0\} \subset \mathcal F_1^- \subset \cdots \subset \mathcal F_g^-.
\end{equation*}
Then $\mathcal B$ and $\mathcal B^-$ are opposite Borel subgroups with common maximal torus $\mathcal T$. We use the Iwahori subgroup
\begin{equation}\label{eq:Iwahori}
    I = \{h \in K \mid \overline h \in \mathcal B^-(\overline{\mathbb F}_p)\},
\end{equation}
where $\overline h$ is the reduction of $h$ modulo $p$. This convention agrees with \cite{shimada_supersingular_2025}. We say that two elements $b, b' \in G(L)$ are $K$-$\sigma$-conjugate if
$b' = h^{-1}b \sigma(h)$ for some $h \in K$. 

We next make the associated root datum explicit. Let $\chi_0 = c_{|_T}$ and, for $1 \leq i \leq g$, let $\chi_i$ be the character of $T$ which sends the element in \eqref{eq:diagonal-torus} to $a_i$. In $X^*(T)\otimes_{\mathbb Z}\mathbb R$, put $e_i = \chi_i - \frac{1}{2} \chi_0$. We denote by $\Phi$ the set of roots. With respect to $\mathcal B$, the positive roots are
\begin{equation*}
    \Phi^+ \coloneq \{e_i - e_j, e_i + e_j \mid 1 \leq i<j \leq g\} \cup \{2e_i \mid 1\leq i \leq g\},    
\end{equation*}
and the simple roots are
\begin{equation*}
    \forall 1 \leq i < g, \qquad \alpha_i \coloneq e_i - e_{i+1} \qquad \alpha_g \coloneq 2e_g.
\end{equation*}
The finite Weyl group $W_0 = N_G(T)(L) / T(L)$, where $N_G(T)$ denotes the normalizer of $T$ in $G$, is isomorphic to the group of \textit{signed permutations} on $g$ letters. By definition, a signed permutation is a bijection of $\{\pm 1, \ldots, \pm g\}$ satisfying $w(-i) = -w(i)$. It is determined by, and written as, the list $[w(1), \ldots, w(g)]$. Its action on the root system is $w(e_i) = \mathrm{sgn}(w(i)) e_{|w(i)|}$. The simple reflection $s_i$ interchanges $i$ and $i+1$ for $i<g$, while $s_g$ changes the sign of $g$. We write $S = \{s_1, \ldots, s_g\}$ for the set of simple reflections of $W_0$. Let $T_{\mathrm{ad}}$ denote the image of $T$ in $G_{\mathrm{ad}}$. We write $\mathscr A$ for the apartment associated with $T_{\mathrm{ad}}$ in the Bruhat-Tits building of $G_{\mathrm{ad}}(L)$. The hyperspecial vertex determined by $\Lambda$ provides an identification
\begin{equation*}
    \mathscr A \simeq X_*(T_{\mathrm{ad}}) \otimes_{\mathbb Z} \mathbb R \simeq \frac{X_*(T) \otimes_{\mathbb Z} \mathbb R}{X_*(Z(G)) \otimes_{\mathbb Z} \mathbb R}.
\end{equation*}
Since $G$ is split, its Iwahori-Weyl group is
\begin{equation*}
    \widetilde{W} = N_G(T)(L)/\mathcal T(\mathcal O_L) \simeq X_*(T) \rtimes W_0.
\end{equation*}
For $\lambda \in X_*(T)$, its image in $\widetilde{W}$ is denoted by $t^{\lambda}$. With this convention, $t^{\lambda} w$ acts on the apartment as $x \mapsto \overline{\lambda} + w(x)$, where $\overline{\lambda}$ is the image of $\lambda$ in $\mathscr A$. The choice of $I$ determines the set of simple reflections $\widetilde{S} = \{s_0, s_1, \ldots, s_g\}$ in the affine Weyl group $W_a \subset \widetilde{W}$, where $s_0$ is the reflection associated to the affine root $\alpha_0 \coloneq 1-2e_1$. For $J \subseteq \widetilde S$, we write $W_J$ for the subgroup of $W_a$ generated by $J$. We have a splitting $\widetilde{W} = W_a \rtimes \Omega$ where $\Omega$ is the stabilizer of the alcove $\mathfrak a \subset \mathscr A$ determined by $I$. The Bruhat order and the length function on $W_a$ are then naturally extended to $\widetilde{W}$. 

The local Hodge cocharacter used in the Shimura datum underlying the Siegel modular space $\mathcal A_g$ is
\begin{equation*}
    \mu: \mathbb G_m \longrightarrow G, \qquad \mu(z) = \mathrm{diag}(\underbrace{z, \ldots , z}_{g \text{ times}}, \underbrace{1, \ldots , 1}_{g \text{ times}}).
\end{equation*}
It has similitude factor $c \circ \mu(z) = z$, and its image $\overline{\mu}$ in $\mathscr A$ satisfies $\langle e_i, \overline{\mu} \rangle = 1/2$ for all $1 \leq i \leq g$. For later use, recall that
\begin{equation*}   
    \mathrm{Adm}(\mu) = \left\{ v \in \widetilde{W} \;\middle|\; v \leq t^{w\cdot \mu} \text{ for some } w \in W_0\right\}.
\end{equation*}
We shall use in particular
\begin{equation*}
    {}^S\mathrm{Adm}(\mu) = {}^S\widetilde{W} \cap \mathrm{Adm}(\mu),
\end{equation*}
where ${}^S\widetilde{W}$ is the set of minimal-length representatives for $W_0 \backslash \widetilde{W}$. Equivalently, it consists of the elements $v$ such that $\ell(sv) = \ell(v)+1$ for every $s\in S$. By \cite{heStratificationsReductionShimura2017} Definition 6.4, Remark 6.5(2), and Theorem 6.10, ${}^S\mathrm{Adm}(\mu)$ is the natural indexing set of the hyperspecial EKOR stratification, which agrees with the Ekedahl-Oort stratification.

Concretely, choose a symplectic trivialization of the Dieudonné lattice of $(A, \lambda) \in \mathcal A_g(\overline{\mathbb F}_p)$ and write its Frobenius as $F = b\sigma$. The index in ${}^S\mathrm{Adm}(\mu)$ of the hyperspecial EKOR stratum it belongs to is the unique $v \in{}^S\mathrm{Adm}(\mu)$ for which the $K$-$\sigma$-conjugacy class of $b$ meets $I \dot v I$. Here $\dot v \in N_G(T)(L)$ denotes a representative of $v$. We now identify this index explicitly in terms of the elementary sequence of Section \ref{Section1}. We introduce
\begin{equation*}
    \omega = [-g, -(g-1), \ldots, -1]\in W_0, \qquad \tau = t^{\mu} \omega \in \widetilde{W}.
\end{equation*}
A direct calculation gives $\tau^{-1}(\alpha_i) = \alpha_{g-i}$ for all $0 \leq i \leq g$. Hence $\tau$ stabilizes the base alcove and belongs to $\Omega$. Let $\varphi$ be an elementary sequence with increment sequence $\delta$. Using the permutation $\pi_{\delta}$ and the signs $\varepsilon_i$ introduced in \eqref{eq:signs}, define
\begin{equation*}
    \forall 1 \leq i \leq g, \qquad \widetilde{\pi}_{\delta}(i) = \varepsilon_i \pi_{\delta}(i) \qquad \widetilde{\pi}_{\delta}(-i) = -\widetilde{\pi}_{\delta}(i),
\end{equation*}
and put $v_{\delta} = t^{\mu} \widetilde{\pi}_{\delta} \in \widetilde{W}$.

\begin{proposition}\label{prop:elementary-affine-label}
The assignment
\begin{equation*}
    \{0, 1\}^g \xrightarrow{\sim} {}^S\mathrm{Adm}(\mu), \qquad \delta \longmapsto v_{\delta},
\end{equation*}
is a bijection. Under the hyperspecial EKOR parametrization, the stratum indexed by $v_{\delta}$ is precisely $\mathcal A_{g,\delta}$.
\end{proposition}

\begin{proof}
Put $J = S\setminus \{s_g\}$, and let ${}^J W_0$ denote the minimal-length representatives for $W_J\backslash W_0$. We claim that
\begin{equation}\label{eq:affine-admissible-parametrization}
    {}^S\mathrm{Adm}(\mu) = {}^S\widetilde{W} \cap W_0 t^{\mu} W_0 = t^{\mu} ({}^J W_0).
\end{equation}
First, according to Corollary 3.12 of \cite{rapoport_guide_2005}, since $\mu$ is minuscule, we have $W_0 \mathrm{Adm}(\mu)W_0 = W_0 t^{\mu} W_0$. Using \cite{heStratificationsReductionShimura2017} Theorem 6.10, we get 
\begin{equation*}
    {}^S\mathrm{Adm}(\mu) = {}^S\widetilde{W} \cap \mathrm{Adm}(\mu) = {}^S\widetilde{W} \cap W_0 \mathrm{Adm}(\mu)W_0 = {}^S\widetilde{W} \cap W_0 t^{\mu} W_0,
\end{equation*}
which is the first equality of \eqref{eq:affine-admissible-parametrization}. For the second equality, let $w \in {}^J W_0$. We have
\begin{equation*}
    (t^{\mu} w)^{-1}(\alpha_i) = \begin{cases}
        w^{-1}(\alpha_i) & \text{if } 1 \leq i < g,\\
        1 + 2w^{-1}(e_g) & \text{if } i = g.    \end{cases}
\end{equation*}
Thus, $(t^{\mu} w)^{-1}(\alpha_i)$ is a positive affine root for every $1 \leq i \leq g$, implying that $t^{\mu}w \in {}^S\widetilde{W}$. It follows that $t^{\mu} ({}^J W_0) \subseteq {}^S\widetilde{W} \cap W_0 t^{\mu} W_0$. Conversely, take $v$ in the right-hand side, and write $v = a t^{\mu}b$ for some $a,b \in W_0$. Let $w \in {}^J W_0$ denote the minimal-length representative of the coset $W_J b$. Thus one may write $b = b' w$ for some $b' \in W_J$. Since $W_J = \mathrm{Stab}_{W_0}(\mu)$, we have $b't^{\mu} = t^{\mu} b'$. Therefore $v = ab' t^{\mu} w$ so that $W_0 v = W_0 t^{\mu} w$. But $v$ is the minimal-length representative of this coset since $v \in {}^S\widetilde{W}$, and so is $t^{\mu} w \in {}^S\widetilde{W}$ by the inclusion we already proved. It follows that $v = t^{\mu} w$, concluding the proof of \eqref{eq:affine-admissible-parametrization}.

For a signed permutation $w$, the conditions $w^{-1}(e_i-e_{i+1}) \in \Phi^+$ for $1 \leq i < g$ mean precisely that its positive values, read in increasing order of their positions, are $1,\ldots, d$, while its negative values are $-g, -(g-1), \ldots, -(d+1)$ for some $0 \leq d \leq g$. Thus an element of ${}^J W_0$ is uniquely determined by the positions of its positive values. By \eqref{eq:pi-definition}, these elements are exactly the $\widetilde{\pi}_{\delta}$ for various $\delta \in \{0, 1\}^{g}$, whose positive positions are given by $P_{\delta}$. Together with \eqref{eq:affine-admissible-parametrization}, this proves the bijection.

To check that the EO strata coincide, choose the symplectic monomial representative $U_{\delta} \in \mathcal G(\mathbb Z_p)$ of $\widetilde{\pi}_{\delta}$ determined by
\begin{equation*}
    \begin{array}{c|cc}
         & U_{\delta} u_i & U_{\delta} u_{-i}\\ \hline
        i \in P & u_{\pi_{\delta}(i)} & u_{-\pi_{\delta}(i)}\\
        i \in Z & u_{-\pi_{\delta}(i)} & -u_{\pi_{\delta}(i)}.
    \end{array}
\end{equation*}
Set $C = \mathrm{diag}(I_g, pI_g)$ and $b_{\delta} = U_{\delta} C$. The operators $F = b_{\delta} \sigma$ and $V = pF^{-1}$ define the structure of a polarized Dieudonné module on $\Lambda_0$. Its reduction modulo $p$ is $M_{\delta}$, under the transformation $u_i \mapsto X_i$ and $u_{-i} \mapsto Y_i$. Indeed, the formulas for $F$ and $V$ are exactly \eqref{eq:normal-P} and \eqref{eq:normal-Z}. Thus, the isomorphism class of $M_{\delta}$ is represented by the $K$-$\sigma$-conjugacy class of $b_{\delta}$ under the map of Remark 8.1 of \cite{viehmann_ekedahloort_2013}. Now consider
\begin{equation*}
    R = \begin{pmatrix}
        0 & -I_g\\
        I_g & 0
        \end{pmatrix}, \qquad h = R U_{\delta} \in K.
\end{equation*}
Since $R$ commutes with $U_{\delta}$, $R^{-1} C R = \mu(p)$, and $\sigma(h) = h$, one has
\begin{equation*}
    h^{-1} b_{\delta} \sigma(h) = \mu(p) U_{\delta}. 
\end{equation*}
But the image of $\mu(p) U_{\delta}$ in $N_G(T)(L)/T(\mathcal O_L) = \widetilde{W}$ is nothing but $v_{\delta}$. It follows that the index of the hyperspecial EKOR stratum corresponding to $\mathcal A_{g,\delta}$ is precisely $v_{\delta}$ as claimed.
\end{proof}

Finally, we compute the length of $v_{\delta}$. The positive roots sent to negative roots by $\widetilde{\pi}_{\delta}$ are precisely $2e_i$ with $i\in Z$, $e_i - e_j$ with $i < j$, $i \in Z$, $j \in P$, and $e_i + e_j$ with $i < j$, $i, j \in Z$. Thus each $i \in Z$ contributes one root $2e_i$ and exactly one root for each $j > i$, giving
\begin{equation}\label{eq:length_of_widetildepi}
    \ell(\widetilde{\pi}_{\delta}) = \sum_{i\in Z} (g+1-i).
\end{equation}
Consider the increment sequence $\delta^c \coloneq (1-\delta_1, \ldots, 1 - \delta_g)$, whose elementary sequence is given by $\varphi^c_i \coloneq i - \varphi_i$ for all $0 \leq i \leq g$. Direct computation shows that 
\begin{equation}\label{eq:definition_widetilde_pideltac}
    \widetilde{\pi}_{\delta^c} = \omega^{-1} \widetilde{\pi}_{\delta}, \text{ and } v_{\delta} = \tau \widetilde{\pi}_{\delta^c}.
\end{equation}
Since $\tau \in \Omega$, applying \eqref{eq:length_of_widetildepi} to $\delta^c$ gives
\begin{equation*}
    \ell(v_{\delta}) = \sum_{j\in P} (g+1-j) = \dim \mathcal A_{g, \delta},
\end{equation*}
agreeing with the formula \eqref{eq:EO-dimension}.

\subsection{Basic affine Deligne-Lusztig varieties and non-emptiness}

We now relate the intersection of an EO stratum with the supersingular locus to the non-emptiness of an affine Deligne-Lusztig variety. For $v \in \widetilde W$ and $b \in G(L)$, the affine Deligne-Lusztig variety attached to $I$ is given on $\overline{\mathbb F}_p$-points by
\begin{equation*}
    X_v(b) = \{hI \in G(L)/I \mid h^{-1}b\sigma(h) \in I \dot v I\},
\end{equation*}
where $\dot v \in N_G(T)(L)$ is any representative of $v$. Writing $[b]$ for the $G(L)$-$\sigma$-conjugacy class of $b$, the definition gives $X_v(b) \not = \emptyset \iff I \dot v I \cap [b] \not = \emptyset$. Recall that $[b]$ is called \textit{basic} if its Newton point is central, see for instance \cite{rapoport_guide_2005} Section 4. For $G = \mathrm{GSp}_{2g}$, this means that the isocrystal $(H \otimes_{\mathbb Q_p} L, b\sigma)$ is isoclinic. Under the identification $\pi_1(G) \simeq \mathbb Z$ induced by the similitude character, the Kottwitz homomorphism is
\begin{equation*}
    \kappa: G(L) \longrightarrow \mathbb Z, \qquad b \longmapsto v_p(c(b)),
\end{equation*}
where $v_p(p) = 1$. We use the same symbol for its induced map on $\widetilde W$. Basic $\sigma$-conjugacy classes for $G$ are uniquely determined by their Kottwitz invariant. Let $J_g$ denote the anti-diagonal identity matrix. The element $\tau = t^\mu \omega$ introduced above has the representative
\begin{equation}
    \dot{\tau} = \begin{pmatrix}
        0 & pJ_g\\
        -J_g & 0
    \end{pmatrix} \in G(\mathbb Q_p).
\end{equation}
It satisfies $c(\dot{\tau}) = p$, $\sigma(\dot{\tau}) = \dot{\tau}$, and $\dot{\tau}^2 = -pI_{2g}$. Thus its isocrystal has only slope $1/2$, and $[\dot{\tau}]$ is the unique basic class with Kottwitz invariant $1$. In particular, a principally polarized abelian variety has Newton class $[\dot{\tau}]$ precisely when the abelian variety is supersingular.

\begin{proposition}\label{prop:dictionary}
    Each of conditions \textup{(i)} and \textup{(ii)} of Theorem \ref{thm:main} is equivalent to
    \begin{equation*}
        X_{v_{\delta}}(\dot{\tau}) \not = \emptyset.
    \end{equation*}
\end{proposition}

\begin{proof}
Choose an integer $N \geq 3$ prime to $p$ and work over $\overline{\mathbb F}_p$ on the finite étale surjective cover $\mathcal A_{g,N} \to \mathcal A_g$ with full symplectic level-$N$ structure. By \cite{gortzFullyHodgeNewton2019} Remark 7.1 and Lemma 7.6, the intersection of an EKOR stratum with a Newton stratum is non-empty if and only if the corresponding affine Deligne-Lusztig variety is non-empty. The hypotheses apply to the Siegel moduli space since $p > 2$. By Proposition \ref{prop:elementary-affine-label}, the EO stratum $\mathcal A_{g,N,\delta}$, pullback of $\mathcal A_{g,\delta}$ to $\mathcal A_{g,N}$, corresponds to the index $v_{\delta} \in {}^S\mathrm{Adm}(\mu)$, while the supersingular locus has Newton class $[\dot{\tau}]$. Consequently,
\begin{equation*}
    \mathcal A_{g,N,\delta} \cap \mathcal S_{g,N} \not = \emptyset \iff X_{v_{\delta}}(\dot{\tau}) \not = \emptyset.
\end{equation*}
Passing along the cover $\mathcal A_{g,N} \to \mathcal A_g$ proves the assertion.
\end{proof}

To decide the non-emptiness of $X_{v_{\delta}}(\dot{\tau})$, we use Proposition 5 of \cite{schremmerNewtonStrataLevi2024}. Before stating it, we need some notation. Every $v \in \widetilde W$ has a unique expression
\begin{equation}\label{eq:LP-factorization}
    v = a t^\lambda c, \qquad \lambda\ \text{is dominant}, \quad a, c \in W_0, \quad t^\lambda c \in {}^S\widetilde W.
\end{equation}
Let $\mathrm{pr}_0 : \widetilde W \to W_0$ denote the projection, so that $\mathrm{pr}_0(v) = ac$. Following \cite{shimada_supersingular_2025} Section 2.5, if $\chi^+: \Phi \to \{0,1\}$ is the indicator function of $\Phi^+$, the set of \textit{length-positive elements} of $v$ is
\begin{equation}\label{eq:LP-definition}
    \mathrm{LP}(v) = \left\{ w \in W_0 \;\middle|\; \langle w\alpha, c^{-1} \lambda \rangle + \chi^+(w\alpha) - \chi^+(acw\alpha) \geq 0 \text{ for every } \alpha \in \Phi^+ \right\}.
\end{equation}
For $x \in W_a$, let $\mathrm{supp}(x) \subseteq \widetilde S$ be the set of simple reflections occurring in a reduced expression of $x$. For an element of $W_0$, its support is similarly a subset of $S$. If $v = x \omega_0$ with $x \in W_a$ and $\omega_0 \in \Omega$, its \textit{$\sigma$-support} $\mathrm{supp}_\sigma(v)$ is the smallest subset of $\widetilde S$ containing $\mathrm{supp}(x)$ and stable under $\mathrm{Ad}(\omega_0)\circ\sigma$. Here $\sigma$ acts trivially on $\widetilde W$, since $G$ is split. The following Theorem was conjectured by Lim in \cite{lim_nonemptiness_2026}, and proved in full generality by Schremmer in \cite{schremmerNewtonStrataLevi2024} Proposition 5. The formulation we give is taken from \cite{shimada_supersingular_2025} Theorem 2.5.

\begin{theorem}\label{thm:schremmer}
    Let $v \in \widetilde W$, and let $b \in G(L)$ be basic with $\kappa(b) = \kappa(v)$. Then $X_v(b) = \emptyset$ if and only if both conditions below hold:
    \begin{equation*}
        |W_{\mathrm{supp}_\sigma(v)}| = \infty \qquad \text{and} \qquad \exists w \in \mathrm{LP}(v) \text{ such that } \mathrm{supp}\left(w^{-1} \mathrm{pr}_0(v) w \right) \subsetneq S.
    \end{equation*}
\end{theorem}

For $v_{\delta} = t^\mu \widetilde{\pi}_{\delta}$, the factorization \eqref{eq:LP-factorization} has $a = 1$, $\lambda = \mu$, and $c = \widetilde{\pi}_{\delta}$. In the next section, we compute both conditions of Theorem \ref{thm:schremmer} uniformly for every $v_{\delta}$, relating them to the cycle graph $\Gamma_{\delta}$.

\section{Proof of the main theorem}\label{Section3}

Throughout this section, we fix an elementary sequence $\varphi$ with increment sequence $\delta$. We write $r_0 \coloneq \min P$, with $r_0 = +\infty$ if $P = \emptyset$.

\subsection{The $\sigma$-support of $v_{\delta}$}

\begin{lemma}\label{lem:affine-support}
    If $P = \emptyset$, then $\mathrm{supp}_\sigma(v_{\delta}) = \emptyset$. Otherwise,
    \begin{equation*}
        \mathrm{supp}_\sigma(v_{\delta}) = \{s_{r_0}, \ldots, s_g\} \cup \{s_0, \ldots, s_{g-r_0}\}.
    \end{equation*}
    Consequently, $|W_{\mathrm{supp}_\sigma(v_{\delta})}| = \infty \iff P \not = \emptyset$ and $2r_0 \leq g+1$.
\end{lemma}

\begin{proof}
Recall the signed permutation $\widetilde{\pi}_{\delta^c} = \omega^{-1} \widetilde{\pi}_{\delta}$ introduced in \eqref{eq:definition_widetilde_pideltac}. We have $v_{\delta} = \tau \widetilde{\pi}_{\delta^c}$. Its entries are
\begin{equation*}
    \widetilde{\pi}_{\delta^c}(i) = \begin{cases}
        -(g+1-\varphi_i) & \text{if } i \in P,\\
        i-\varphi_i & \text{if } i \in Z.
    \end{cases}
\end{equation*}
If $P = \emptyset$, then $\widetilde{\pi}_{\delta^c} = 1$ and $v_{\delta} = \tau$, proving the first assertion. Suppose that $P \not = \emptyset$. The element $\widetilde{\pi}_{\delta^c}$ fixes $1, \ldots, r_0-1$ and preserves $\{\pm r_0, \ldots, \pm g\}$, so its support is contained in $\{s_{r_0}, \ldots, s_g\}$. Moreover, $\widetilde{\pi}_{\delta^c}(r_0) = -g$. For $r_0 \leq j < g$, every element of $W_{S\setminus\{s_j\}}$ preserves $\{1, \ldots, j\}$, whereas $\widetilde{\pi}_{\delta^c}$ does not. Also, $\widetilde{\pi}_{\delta^c} \not \in W_{S\setminus\{s_g\}}$, since it has a negative entry. Thus
\begin{equation*}
    \mathrm{supp}(\widetilde{\pi}_{\delta^c}) = \{s_{r_0}, \ldots, s_g\}.
\end{equation*}

Write $v_{\delta} = x_{\delta}\tau$, where $x_{\delta} = \tau \widetilde{\pi}_{\delta^c}\tau^{-1} \in W_a$. Since $\mathrm{Ad}(\tau)(s_i) = s_{g-i}$ and $\sigma$ acts trivially, the smallest $\mathrm{Ad}(\tau)$-stable subset containing $\mathrm{supp}(x_{\delta})$ is therefore the union $\{s_{r_0}, \ldots, s_g\} \cup \{s_0, \ldots, s_{g-r_0}\}$. This union is all of $\widetilde S$ exactly when $2r_0 \leq g+1$. Since every proper standard parabolic subgroup of the affine Weyl group of type $\widetilde C_g$ is finite, the equivalence stated in the Lemma follows.
\end{proof}

The complementary case has a simple interpretation in terms of bad cycles.

\begin{lemma}\label{lem:right-half}
If $W_{\mathrm{supp}_\sigma(v_{\delta})}$ is finite, then $\mathcal B_{\delta} = \emptyset$.
\end{lemma}

\begin{proof}
By Lemma \ref{lem:affine-support}, our hypothesis says that $P = \emptyset$ or $2r_0 > g+1$. In the former case the assertion is immediate, so let us assume the latter. We write
$h \coloneq \lfloor g/2 \rfloor$ and
\begin{equation*}
    L_0 = \{1, \ldots, h\}, \qquad R_0 = \{g-h+1, \ldots, g\}.
\end{equation*}
Note that $L_0$ and $R_0$ are disjoint, and that their union is the whole of $\{1, \ldots , g\}$ if $g$ is even, and misses only $h+1$ if $g$ is odd. Our assumption implies that $P \subseteq R_0$. All integers preceding $R_0$ belong to $Z$, so \eqref{eq:pi-definition} gives $\pi_{\delta}(i) = g+1-i$ for $i\in L_0$. If $g$ is odd, then $h+1$ is fixed by $\pi_{\delta}$. Since $\pi_{\delta}$ is a permutation, it follows that
\begin{equation*}
    \pi_{\delta}(L_0) = R_0, \qquad \pi_{\delta}(R_0) = L_0.
\end{equation*}
Every cycle of $\pi_{\delta}$ meeting $P$ therefore alternates between $L_0$ and $R_0$. Between two consecutive elements of $P$ in such a cycle, the number of elements of $Z$ is odd. The definition \eqref{eq:orientation} then prevents any orientation on such a cycle $C$ from being positive at every value of $C \cap P$. Thus no cycle is bad.
\end{proof}

\subsection{Length-positive elements and the cycle graph}

We now consider the second condition of Theorem \ref{thm:schremmer}. For $w \in W_0$, write
\begin{equation*}
    w^{-1} = [\xi_1, \ldots, \xi_g].
\end{equation*}

\begin{lemma}\label{lem:LP-explicit}
One has $w \in \mathrm{LP}(v_{\delta})$ if and only if
    \begin{align*}
        \forall i \in P, & & & \xi_i > 0, \\
        \forall i \in P, \forall j \in Z, & & & (i < j \text{ and } \xi_j > 0) \implies \xi_i < \xi_j.
    \end{align*}
\end{lemma}

\begin{proof}
For $\beta \in \Phi$, we write
\begin{equation}\label{eq:E-beta}
    E_{\delta}(\beta) \coloneq \langle \beta, \widetilde{\pi}_{\delta}^{-1} \mu \rangle + \chi^+(\beta) - \chi^+(\widetilde{\pi}_{\delta}\beta).
\end{equation}
The coordinates of $\widetilde{\pi}_{\delta}^{-1} \mu$ in the apartment $\mathscr A$ are $(\varepsilon_1/2, \ldots, \varepsilon_g/2)$. For $\beta\in\Phi^+$, the last two terms of \eqref{eq:E-beta} give $1$ if $\widetilde{\pi}_{\delta} \beta$ is negative, and $0$ otherwise. Recall that $\pi_{\delta}$ is increasing on $P$, decreasing on $Z$, and that every value on $P$ is smaller than every value on $Z$. Consequently, $\widetilde{\pi}_{\delta}(2e_i)$ is negative exactly when $i\in Z$. For $i<j$, the image of $e_i - e_j$ is negative exactly when $i \in Z$ and $j \in P$, whereas the image of $e_i+e_j$ is negative exactly when $i, j \in Z$. Using the coordinates of $\widetilde{\pi}_{\delta}^{-1} \mu$ and the identity $\varepsilon_i = 2\delta_i-1$, we therefore obtain
\begin{align*}
    E_{\delta}(2e_i) & = \varepsilon_i+(1-\delta_i) = \delta_i,\\
    E_{\delta}(e_i-e_j) & = \frac{\varepsilon_i - \varepsilon_j}{2} + (1-\delta_i)\delta_j = \delta_i (1-\delta_j),\\
    E_{\delta}(e_i+e_j) & = \frac{\varepsilon_i + \varepsilon_j}{2} + (1-\delta_i)(1-\delta_j) = \delta_i\delta_j.
\end{align*}
Here $1 \leq i\leq g$ in the first line, and $1\leq i < j \leq g$ in the last two. In particular, $E_{\delta}$ takes only the values $0$ and $1$ on $\Phi^+$. By definition, we have $w \in \mathrm{LP}(v_{\delta})$ if and only if $E_{\delta}(w\alpha) \geq 0$ for every $\alpha \in \Phi^+$. Since $E_{\delta}(-\beta) = -E_{\delta}(\beta)$, this is equivalent to requiring that $w^{-1}(\beta) \in \Phi^+$ for every positive root $\beta$ with $E_{\delta}(\beta) = 1$. These roots are precisely
\begin{equation*}
    \begin{array}{ll}
        2e_i & \text{for } i\in P,\\
        e_i + e_j & \text{for } i<j \text{ with } i,j \in P,\\
        e_i - e_j & \text{for } i<j, \text{ with } i \in P \text{ and } j \in Z.
    \end{array}
\end{equation*}
The first family gives that $\xi_i > 0$ for all $i \in P$, which also ensures positivity for the second family. For the third family, if $\xi_j < 0$, then $w^{-1}(e_i-e_j) = e_{\xi_i} + e_{|\xi_j|}$ is positive. If $\xi_j > 0$, its positivity is equivalent to $\xi_i < \xi_j$. This can be summed up by the implication $(i < j \text{ and } \xi_j > 0) \implies \xi_i < \xi_j$. 
\end{proof}

We also recall the following description of elements $w \in W_0$ whose support is proper.

\begin{lemma}\label{lem:finite-parabolic}
    For $w \in W_0$, one has $\mathrm{supp}(w) \subsetneq S$ if and only if $\{1, \ldots, k\}$ is preserved by $w$ for some $1 \leq k \leq g$.
\end{lemma}

\begin{proof}
The signed permutations preserving $\{1, \ldots, k\}$ are exactly the elements of $W_{S\setminus\{s_k\}}$. Indeed, they permute the first $k$ integers without changing signs and act by an arbitrary signed permutation on the remaining integers. Every proper standard parabolic subgroup is contained in one of these maximal standard parabolic subgroups.
\end{proof}

The next proposition translates the second condition of Theorem \ref{thm:schremmer} into the condition of Lemma \ref{lem:no-incoming}, assuming that $W_{\mathrm{supp}_\sigma(v_{\delta})}$ is infinite.

\begin{proposition}\label{prop:bridge}
    Suppose that $W_{\mathrm{supp}_\sigma(v_{\delta})}$ is infinite. There exists $w \in \mathrm{LP}(v_{\delta})$ such that $\mathrm{supp}(w^{-1}\widetilde{\pi}_{\delta}w) \subsetneq S$ if and only if there is a non-empty subset $\mathcal U \subseteq \mathcal B_{\delta}$ such that for every arrow $C \longrightarrow C'$ in $\Gamma_{\delta}$, we have $(C' \in \mathcal U \implies C \in \mathcal U)$.
\end{proposition}

\begin{proof}
The hypothesis that $W_{\mathrm{supp}_\sigma(v_{\delta})}$ is infinite is equivalent, by Lemma \ref{lem:affine-support}, to the fact that $P \not = \emptyset$ and $2r_0 \leq g+1$. For $w^{-1} = [\xi_1,\ldots,\xi_g]$, we shall use the simple identity
\begin{equation}\label{eq:conjugate-sign}
    (w^{-1} \widetilde{\pi}_{\delta} w)(|\xi_i|) = \mathrm{sgn}(\xi_i) \varepsilon_i \xi_{\pi_{\delta}(i)}.
\end{equation}
First suppose that $\mathcal U$ is as in the Proposition. Recall that every cycle $C \in \mathcal B_{\delta}$, and so a fortiori in $\mathcal U$, comes with a distinguished orientation $t_C$. We define
\begin{equation*}
    S_{\mathcal U} \coloneq \bigcup_{C \in \mathcal U} C, \qquad k \coloneq |S_{\mathcal U}|, \qquad m \coloneq |S_{\mathcal U} \cap P|.
\end{equation*}
Choose a bijection $\rho : \{1, \ldots, g\} \xrightarrow{\sim} \{1, \ldots, g\}$ such that
\begin{align*}
    \rho(S_{\mathcal U}\cap P) & = \{1, \ldots, m\},\\
    \rho(S_{\mathcal U}\cap Z) & = \{m+1, \ldots, k\},\\
    \rho(\{1, \ldots, g\}\setminus S_{\mathcal U}) & = \{k+1, \ldots, g\}.
\end{align*}
Define
\begin{equation*}
    \xi_i = \begin{cases}
        t_C(i) \rho(i), & \text{if } i \in C \text{ for some } C \in \mathcal U,\\
        \varepsilon_i \rho(i), & \text{if } i \not \in S_{\mathcal U},
    \end{cases}
\end{equation*}
and let $w \in W_0$ be the signed permutation determined by $w^{-1} = [\xi_1,\ldots,\xi_g]$. 

Every $\xi_i$ with $i \in P$ is positive. Suppose that $i < j$, with $i \in P$, $j \in Z$, and $\xi_j > 0$. Then $j$ lies in a cycle $C' \in \mathcal U$ with $t_{C'}(j) = 1$. If $i \not \in S_{\mathcal U}$, its cycle would give an arrow into $C'$ from outside $\mathcal U$, contrary to our assumptions. Thus $i \in S_{\mathcal U}$, and by construction of $\rho$ we obtain $\xi_i < \xi_j$. Lemma \ref{lem:LP-explicit} then proves that $w \in \mathrm{LP}(v_{\delta})$. Moreover, the definition \eqref{eq:orientation} of orientations and the relation \eqref{eq:conjugate-sign} give
\begin{equation*}
    \forall i \in S_{\mathcal U}, \qquad (w^{-1}\widetilde{\pi}_{\delta} w)(|\xi_i|)
    =|\xi_{\pi_{\delta}(i)}|.
\end{equation*}
Hence $w^{-1} \widetilde{\pi}_{\delta} w$ preserves $\{1,\ldots,k\}$,
and Lemma \ref{lem:finite-parabolic} ensures that $\mathrm{supp}(w^{-1}\widetilde{\pi}_{\delta}w) \subsetneq S$.

Conversely, choose $w \in \mathrm{LP}(v_{\delta})$ with $\mathrm{supp}(w^{-1}\widetilde{\pi}_{\delta} w) \subsetneq S$. Choose $k$ as in Lemma \ref{lem:finite-parabolic}, that is such that $w^{-1}\widetilde{\pi}_{\delta} w$ preserves $\{1, \ldots, k\}$. We consider
\begin{equation*}
    S_0 \coloneq \{i \mid |\xi_i| \leq k\},
\end{equation*}
and we define $t(i) \coloneq \mathrm{sgn}(\xi_i)$ for all $1 \leq i \leq g$. By \eqref{eq:conjugate-sign}, we have 
\begin{equation*}
    \forall i \in S_0, \qquad (w^{-1} \widetilde{\pi}_{\delta} w)(|\xi_i|) = t(i) \varepsilon_i t(\pi_{\delta}(i)) |\xi_{\pi_{\delta}(i)}| \in \{1, \ldots, k\}.
\end{equation*}
It follows that $\pi_{\delta}(i) \in S_0$, and therefore $S_0$ is a union of $\pi_{\delta}$-cycles, and also that
\begin{equation}\label{eq:selected-orientation}
    \forall i \in S_0, \qquad t(\pi_{\delta}(i)) = \varepsilon_i t(i).
\end{equation}
Since $t(i) = 1$ for all $i \in P$ by Lemma \ref{lem:LP-explicit}, every cycle contained in $S_0$ and meeting $P$ is bad. 

We claim that $S_0 \cap P \not = \emptyset$. Otherwise $S_0 \subseteq Z$, and $\pi_{\delta}$ restricts to a strictly decreasing bijection of $S_0$, hence an involution. The relation \eqref{eq:selected-orientation} shows that $t$ changes sign along the cycles within $S_0$, so some $j \in S_0$ must satisfy $t(j) = 1$. If $j > r_0$, then Lemma \ref{lem:LP-explicit} gives $0 < \xi_{r_0} < \xi_j \leq k$, contradicting the fact that $S_0 \subseteq Z$. Thus $j < r_0$, and
\begin{equation*}
    \pi_{\delta}(j) = g+1-j > r_0,
\end{equation*}
where the inequality uses $2r_0 \leq g+1$. Since $\pi_{\delta}(j) \in S_0 \subseteq Z$, the formula \eqref{eq:pi-definition} now gives
\begin{equation*}
    \pi_{\delta}^2(j) = j + \varphi_{\pi_{\delta}(j)}.
\end{equation*}
But since $\pi_{\delta}(j) > r_0$, we have $\varphi_{\pi_{\delta}(j)} > 0$ and therefore $\pi_{\delta}^2(j) > j$, contradicting the fact that $\pi_{\delta}^2$ is the identity on $S_0$.

It follows that
\begin{equation*}
    \mathcal U \coloneq \{C \subseteq S_0 \mid C \text{ is a cycle of } \pi_{\delta} \text{ and } C \cap P \not = \emptyset\},
\end{equation*}
is a non-empty subset of $\mathcal B_{\delta}$. Now suppose that $C \to C'$ is an arrow in $\Gamma_{\delta}$ with $C' \in \mathcal U$, and choose $i \in C \cap P$ and $j \in C' \cap Z$ as in Definition \ref{defi:cycle_graph}. The restriction of $t$ to $C'$ is its distinguished orientation, so $\xi_j > 0$. By Lemma \ref{lem:LP-explicit}, $0 < \xi_i < \xi_j \leq k$. Thus $i\in S_0$, and the entire cycle $C$ lies in $S_0$. Since $i\in P$, we have $C \in \mathcal U$, as required.
\end{proof}

\begin{proof}[Proof of Theorem \ref{thm:main}]
The element $\dot{\tau}$ is basic, and $\kappa(v_{\delta}) = \kappa(\dot\tau) = 1$, so Theorem \ref{thm:schremmer} applies. If $W_{\mathrm{supp}_\sigma(v_{\delta})}$ is finite, then Theorem \ref{thm:schremmer} shows that $X_{v_{\delta}}(\dot{\tau}) \not = \emptyset$. In this case, Lemma \ref{lem:right-half} gives $\mathcal B_{\delta} = \emptyset$, so $(*)_{\delta}$ also holds.

Otherwise, we may assume that $W_{\mathrm{supp}_\sigma(v_{\delta})}$ is infinite. Since $\mathrm{pr}_0(v_{\delta}) = \widetilde{\pi}_{\delta}$, Proposition \ref{prop:bridge} identifies the second condition of Theorem \ref{thm:schremmer} with the existence of a non-empty subset $\mathcal U \subseteq \mathcal B_{\delta}$ of bad cycles to which no external arrow of $\Gamma_{\delta}$ points. By Lemma \ref{lem:no-incoming}, this is exactly the failure of $(*)_{\delta}$. Thus in all cases, the condition $(*)_{\delta}$ is equivalent to $X_{v_{\delta}}(\dot\tau)\not = \emptyset$. Proposition \ref{prop:dictionary} completes the proof.
\end{proof}

\section{The number of EO strata meeting the supersingular locus}\label{sec:enumeration}

Recall that $N_g$ is the number of EO strata in $\mathcal A_g$ meeting $\mathcal S_g$. Equivalently, it is the number of increment sequences $\delta \in \{0, 1\}^{g}$ such that $(*)_{\delta}$ holds. By Lemma \ref{lem:positive-p-rank}, any such $\delta$ actually belongs to $\mathcal D_g$ where
\begin{equation}\label{eq:Dg}
    \mathcal D_g \coloneq \left\{ \delta \in \{0,1\}^g \mid \delta_1 = 0 \right\}, \qquad |\mathcal D_g|=2^{g-1}.
\end{equation}
We write
\begin{equation}\label{eq:defect-count}
    f_g \coloneq 2^{g-1}-N_g \geq 0.
\end{equation}
Thus $f_g$ counts the increment sequences $\delta \in \mathcal D_g$ for which $(*)_{\delta}$ fails. Equivalently, it counts the number of $p$-rank zero EO strata which do not meet the supersingular locus. The aim of this section is to prove the following estimate.

\begin{theorem}\label{thm:asymptotic-count}
For all $g \geq 3$, we have
    \begin{equation*}
        \frac{2^{g-2}-1}{g-2} \leq f_g \leq \frac{2^{g+2}-2}{g+2}.
    \end{equation*}
In particular, $N_g \sim 2^{g-1}$ as $g \to \infty$.
\end{theorem}

We first build a particular bijection between $\mathcal D_g$ and the set of compositions of $g$, and then compare $f_g$ with the number of compositions whose first part is maximal.

\subsection{A parametrization by compositions of integers}

A composition of $n \geq 1$ is a finite sequence of positive integers with sum $n$. Fix an elementary sequence $\varphi$ whose increment sequence $\delta$ belongs to $\mathcal D_g$. Recall that $|P| = \varphi_g$ and that for all $i \in P$, $\pi_{\delta}(i) = \varphi_i < i$. It follows that every cycle of $\pi_{\delta}$ must meet $Z$ non-trivially. Besides, recall that $\pi_{\delta}(P) = \{1, \ldots, \varphi_g\}$. For each $1 \leq j \leq g-\varphi_g$, define

\begin{equation}\label{eq:definition_chains}
    a_j \coloneq \min \{a \geq 1 \mid \pi_{\delta}^{a-1}(\varphi_g+j) \in Z\}, \qquad T_j \coloneq \left( \pi_{\delta}^h(\varphi_g+j) \right)_{0 \leq h < a_j}.
\end{equation}

We call $T_j$ the $j$-th \textit{chain} of $\pi_{\delta}$. It is a finite decreasing sequence of $a_j$ distinct integers. Its first $a_j-1$ entries belong to $P$, and its last entry belongs to $Z$. In particular, if $\varphi_g+j \in Z$, then $a_j = 1$ and $T_j = (\varphi_g+j)$. Every integer of $\{1, \ldots , g\}$ belongs to exactly one of the chains $T_j$, so that all together they form a partition of $\{1, \ldots, g\}$.  Thus, we have $\sum_{j=1}^{g-\varphi_g} a_j = g$ and
\begin{equation}\label{eq:chain-composition}
    c(\delta) \coloneq (a_1, \ldots, a_{g-\varphi_g}),
\end{equation}
is a composition of $g$. Before explaining the inverse map, we record a useful property of the construction $\delta \mapsto c(\delta)$.

We define a function $r$ on $\{1, \ldots, g\}$ by the formula
\begin{equation}\label{eq:chain-position}
    \forall 1 \leq j \leq g - \varphi_g, \; \forall 0 \leq h < a_j, \qquad r \left( \pi_{\delta}^h(\varphi_g+j) \right) \coloneq h.
\end{equation}
In other words, $r(i)$ is the position of $i$ within the chain $T_j$ to which it belongs. In particular, $r(\varphi_g+j) = 0$ and $r(\pi_{\delta}(i)) = r(i)+1$ for $i \in P$. 

\begin{lemma}\label{lem:run-length-order}
    We have $r(1) \geq \cdots \geq r(g)$.
\end{lemma}

\begin{proof}
For every $h \geq 0$, we define $H_h \coloneq \{i \mid r(i) \geq h\}$. We claim that 
\begin{equation*}
    H_0 = \{1, \ldots, g\}, \text{ and } \forall h \geq 0, \quad H_{h+1} = \pi_{\delta}(P \cap  H_h).
\end{equation*}
The equality for $H_0$ is clear. On the other hand, let $i \in P \cap H_h$. We have $r(\pi_{\delta}(i)) = r(i) + 1 \geq h+1$ so that $\pi_{\delta}(i) \in H_{h+1}$. Conversely, if $i \in H_{h+1}$ then $r(i) \geq 1$ so $i$ is not the first element in the chain $T_j$ to which it belongs. Thus $\pi_{\delta}^{-1}(i) \in P$ and we have $r(\pi_{\delta}^{-1}(i)) = r(i) - 1 \geq h$ so that $i \in \pi_{\delta}(P \cap H_{h})$. This proves the claim.

Next, we prove that every $H_h$ has the form $\{1, \ldots , m_h\}$ for some $m_h \geq 0$. This already holds for $h = 0$ with $m_0 = g$. Assume that it holds for some $h \geq 0$. Let us write $P = \{p_1, \ldots , p_{\varphi_g}\}$ where $p_1 < \ldots < p_{\varphi_g}$. Then $P \cap H_h = \{p_1, \ldots , p_k\}$ for some $0 \leq k \leq \varphi_g$. Since $\pi_{\delta}(p_i) = i$ for every $1 \leq i \leq \varphi_g$, we have $H_{h+1} = \pi_{\delta}(P \cap H_h) = \{1, \ldots , k\}$. Taking $m_{h+1} = k$, this proves the claim. 

We can now conclude. Let $1 \leq i < j \leq g$. Clearly, we have $j \in H_{r(j)}$, from which it follows that $i \in H_{r(j)}$. This means that $r(i) \geq r(j)$, and concludes the proof.
\end{proof}

We now describe the construction inverse of $\delta \mapsto c(\delta)$. For a composition $c = (a_1, \ldots, a_s)$ of $g$, consider the set
\begin{equation*}
    \mathcal H(c) \coloneq \{(j, h) \mid 1 \leq j \leq s \text{ and } 0 \leq h < a_j\}.
\end{equation*}
We equip $\mathcal H(c)$ with a total order as follows: for $(j, h), (j', h') \in \mathcal H(c)$,
\begin{equation*}
    (j, h) \leq (j', h') \iff h > h' \text{ or } (h = h' \text{ and } j \leq j').
\end{equation*}
In other words, the order is decreasing with $h$, and then increasing with $j$. Assign the value $0$ to $(j, a_j-1) \in \mathcal H(c)$ for every $1 \leq j \leq s$, and the value $1$ to every other pair. Let $\delta(c)$ be the sequence obtained by reading these values in the order defined by $\leq$. 

Equivalently, the set $\mathcal H(c)$ can be represented by drawing $s$ columns of respective heights $a_1, \ldots, a_s$, aligned at the bottom. The top of each column is assigned $0$ and the other entries are assigned $1$. Then $\delta(c)$ is obtained by reading each value row by row, from top to bottom and from left to right. For example, $c = (3, 1, 2)$ gives
\begin{equation*}
    \begin{array}{ccc}
        0 & &\\
        1 & & 0\\
        1 & 0 & 1
    \end{array} \qquad \mapsto \qquad \delta(c) = (0, 1, 0, 1, 0, 1).
\end{equation*}

\begin{proposition}\label{prop:composition-parametrization}
    The assignments $\delta \mapsto c(\delta)$ and $c \mapsto \delta(c)$ are inverse bijections between $\mathcal D_g$ and the set of compositions of $g$.
\end{proposition}

\begin{proof}
For $c = (a_1, \ldots, a_s)$ a composition of $g$, the minimum of $\mathcal H(c)$ lies at the top of its corresponding column, so that it is assigned $0$. Thus $\delta(c) \in \mathcal D_g$ since it starts with $\delta(c)_1 = 0$. Let $P$, $Z$ and $\pi_{\delta(c)}$ be defined with respect to $\delta(c)$. The map $(j, h) \mapsto (j, h+1)$ is a well-defined increasing bijection from the subset of pairs $(j, h) \in \mathcal H(c)$ assigned with $1$, onto the subset of pairs $(j, h) \in \mathcal H(c)$ with $h > 0$. Consider the increasing bijection 
\begin{equation}\label{eq:map_iota_c}
    \iota_c: \mathcal H(c) \xrightarrow{\sim} \{1, \ldots, g\},
\end{equation}
induced by the total order $\leq$ on $\mathcal H(c)$. By definition, $\iota_c$ sends the pairs $(j, h)$ assigned with $1$ to $P$, and the pairs $(j, h)$ with $h > 0$ to $\{1, \ldots , g-s\}$. Thus, the mapping $(j, h) \mapsto (j, h+1)$ coincides, under the identification $\iota_c$, with the restriction of $\pi_{\delta(c)}$ to $P$. We have $|P| = g-s$, and it follows that $\iota_c(j,h) = \pi_{\delta(c)}^{h}(g-s+j)$. The columns of $\mathcal H(c)$ are the chains of $\pi_{\delta(c)}$ and we obtain $c(\delta(c)) = c$. 

Conversely, start with $\delta \in \mathcal D_g$ and consider $\mathcal H(c(\delta))$. The partition of $\{1, \ldots, g\}$ by the chains $T_j$ of $\pi_{\delta}$ yields a bijection
\begin{equation*}
    f: \mathcal H(c(\delta)) \xrightarrow{\sim} \{1, \ldots, g\}, \qquad (j, h) \mapsto \pi_{\delta}^h(\varphi_g + j).
\end{equation*}
We prove that this map is increasing. If $h > h'$, then by \eqref{eq:chain-position} we have $r(f(j, h)) = h > h' = r(f(j', h'))$. By Lemma \ref{lem:run-length-order}, this forces $f(j, h)  < f(j', h')$. Now if $h = h'$ and $j < j'$, we have $\varphi_g + j < \varphi_g + j'$ and since $h < \min(a_j, a_{j'})$, both $\pi_{\delta}^{\nu}(\varphi_g + j)$ and $\pi_{\delta}^{\nu}(\varphi_g + j')$ belong to $P$ for all $0 \leq \nu < h$. As $\pi_{\delta}$ is increasing on $P$, we deduce that $f(j, h) < f(j', h)$. Thus $f$ is increasing and therefore $f = \iota_{c(\delta)}$. We deduce that the permutations $\pi_{\delta}$ and $\pi_{\delta(c(\delta))}$ produce the same chains, and therefore $\delta = \delta(c(\delta))$. 
\end{proof}

\subsection{Comparison with constrained compositions}

For $n \geq 1$, let $\mathcal W_n$ be the set of compositions of $n$ whose first part is weakly maximal, and put $W_n = |\mathcal W_n|$. We compare $\mathcal W_n$ with the set of sequences $\delta \in \mathcal D_g$ failing $(*)_{\delta}$. We first record two lemmas.

\begin{lemma}\label{lem:cycle-rigidity}
    Let $\delta, \delta' \in \mathcal D_g$. Suppose that a bijection $f : \{1, \ldots, g\} \to \{1, \ldots, g\}$ satisfies
    \begin{equation*}
        f \circ \pi_{\delta} = \pi_{\delta'} \circ f, \text{ and } \forall 1 \leq i \leq g, \quad \delta'_{f(i)} = \delta_i.
    \end{equation*}
    Then $\delta = \delta'$.
\end{lemma}

\begin{proof}
Let $P'$ and $Z'$ be the two subsets of $\{1, \ldots, g\}$ defined by $\delta'$. By hypothesis, $f$ maps $P$ and $Z$ bijectively and respectively to $P'$ and to $Z'$. Given two integers $i, j \in \{1, \ldots, g\}$, we claim that
\begin{equation*}
    (i-j)(f(i)-f(j)) < 0 \implies \delta_i = \delta_j.
\end{equation*}
Writing $a = \pi_{\delta}^{-1}(i)$ and $b = \pi_{\delta}^{-1}(j)$, the inequality can be rewritten as
\begin{equation*}
    (\pi_{\delta}(a) - \pi_{\delta}(b))(\pi_{\delta'}(f(a)) - \pi_{\delta'}(f(b))) < 0.
\end{equation*}
Since $\pi_{\delta}(P)$ precedes $\pi_{\delta}(Z)$, and likewise $\pi_{\delta'}(P')$ precedes $\pi_{\delta'}(Z')$, this inequality forces $\delta_a = \delta_b$. Thus we have $\delta_a = \delta_b = \delta'_{f(a)} = \delta'_{f(b)}$. Since $\pi_{\delta}$ and $\pi_{\delta'}$ have the same monotonicity on $P$ and $P'$, and on $Z$ and $Z'$, we deduce that $(a-b)(f(a)-f(b)) < 0$. Iterating, it follows that actually $\delta_{\pi_{\delta}^{-\nu}(i)} = \delta_{\pi_{\delta}^{-\nu}(j)}$ for every $\nu \geq 1$. Taking $\nu$ as the order of $\pi_{\delta}$, it follows that $\delta_i = \delta_j$ as claimed. 

If $i \in P$ and $j \in Z$, since $\delta_i \not = \delta_j$, the claim implies that 
\begin{equation*}
    i < j \iff f(i) < f(j).
\end{equation*}
Suppose that $\delta \not = \delta'$, and let $k$ be the least index where they differ. Exchanging $\delta, \delta'$ and replacing $f$ with $f^{-1}$ if necessary, we can assume that $\delta_k = 1$ and $\delta'_k = 0$. Since the sequences agree before $k$, we have
\begin{equation*}
    \begin{aligned}
        |P \cap \{1, \ldots, k\}| & = |P' \cap \{1, \ldots, k\}| + 1,\\
        |Z \cap \{1, \ldots, k\}| & = |Z' \cap \{1, \ldots, k\}| - 1.
    \end{aligned}
\end{equation*}
The first equality and the bijection $f: P\to P'$ imply that some $i \in P$ satisfies $i \leq k < f(i)$. The second equality and the bijection $f: Z \to Z'$ imply that some $j \in Z$ satisfies $f(j) \leq k < j$. Thus $i < j$ but $f(i) > f(j)$, which is a contradiction.
\end{proof}

\begin{lemma}\label{lem:boundary-vertex}
    If $\delta \in \mathcal D_g$ fails $(*)_{\delta}$, then $P \not = \emptyset$ and there exist a bad cycle $C$ of $\pi_{\delta}$ and an integer $j \in C \cap Z$ such that 
    \begin{equation*}
        t_C(j) = 1 \text{ and } j < \min P.
    \end{equation*}
\end{lemma}

\begin{proof}
Since $(*)_{\delta}$ fails, $\pi_{\delta}$ must have at least one bad cycle and therefore $P \not = \emptyset$. Let $i_0 = \min P$ and let $C_0$ be the cycle of $i_0$. We have $\pi_{\delta}(i_0) = 1$, and if $C_0$ is bad, then $t_{C_0}(1) = t_{C_0}(i_0) = 1$. Since $\delta_1 = 0$, the integer $j = 1$ has the required properties.

On the other hand, if $C_0$ is good, choose a bad cycle $C$ unreachable from any good cycle. Recall from \eqref{eq:definition_chains} the definition of the chains of $\pi_{\delta}$. Choose such a chain containing an element of $C \cap P$, and let $j$ be the last integer of that chain. The entire chain belongs to $C$, and $t_C$ takes the constant value $1$ on this chain. Thus $j \in C \cap Z$ and $t_C(j) = 1$. If we had $i_0 < j$, then Definition \ref{defi:cycle_graph} would give an arrow $C_0 \to C$, which is a contradiction. As $j \not = i_0$, we obtain $j < i_0$ as desired.
\end{proof}

\begin{proposition}\label{prop:failure-sandwich}
    For every $g \geq 3$,
    \begin{equation*}
        W_{g-2} \leq f_g \leq W_{g+1}.
    \end{equation*}
\end{proposition}

\begin{proof}
For the lower bound, we construct an injective map from $\mathcal W_{g-2}$ into the subset of $\mathcal D_{g}$ consisting of those $\delta$ failing $(*)_{\delta}$. Write a composition in $\mathcal W_{g-2}$ in the form $(a, R)$, where $a$ is its first component and $R$, potentially empty, is a composition of $g-2-a$. We consider
\begin{equation*}
    c \coloneq (a+1, R, 1), \qquad \delta \coloneq \delta(c).
\end{equation*}
Let us prove that $(*)_{\delta}$ fails. Recall the map $\iota = \iota_{c} : \mathcal H(c) \xrightarrow{\sim} \{1, \ldots, g\}$ defined in \eqref{eq:map_iota_c}. Since every component of $R$ is at most $a$, the first column of $\mathcal H(c)$ is the unique tallest column. The top of this column is $(1,a)$ and therefore satisfies $\iota(1,a) = 1$. The last column has height $1$, it consists of the pair $(g-\varphi_g, 0)$, where $\varphi$ is the elementary sequence of $\delta$, and it satisfies $\iota(g - \varphi_g, 0) = g$. Since both pairs are assigned a $0$ in $\mathcal H(c)$, we have $1, g \in Z$, and therefore \eqref{eq:pi-definition} gives
\begin{equation*}
    \pi_{\delta}(1) = g, \qquad \pi_{\delta}(g) = \varphi_g+1.
\end{equation*}
Now, by the proof of Proposition \ref{prop:composition-parametrization}, the equality $\iota(1,a) = 1$ means that $\pi_{\delta}^a(\varphi_g + 1) = 1$. Thus, the subset 
\begin{equation*}
    C \coloneq \{1, g, \varphi_g +1, \pi_{\delta}(\varphi_g + 1), \ldots , \pi_{\delta}^{a-1}(\varphi_g+1)\} \subseteq \{1, \ldots, g\},
\end{equation*}
forms a cycle of $\pi_{\delta}$. It is the union of the first chain $T_1 = (\pi_{\delta}^{h}(\varphi_g + 1))_{0 \leq h \leq a}$ and of the last chain $T_{g - \varphi_g} = (g)$. Thus $C$ contains $a$ elements of $P$ and precisely two elements of $Z$, namely $1$ and $g$. We define an orientation on $C$ by assigning $+1$ on $C \cap P$ and at $1$, and $-1$ at $g$. Clearly, $C$ equipped with this orientation is a bad cycle. Its only element of $Z$ with orientation $+1$ is the integer $1$, therefore no arrow of $\Gamma_{\delta}$ can point to $C$. It follows that $(*)_{\delta}$ fails.

By Proposition \ref{prop:composition-parametrization}, one can recover $c$ from $\delta$, and then recover the initial composition $(a, R)$ by removing the last component and decreasing the first one by $1$, justifying the injectivity. This proves the lower bound. 

For the upper bound, we construct an injective map from the set of sequences $\delta \in \mathcal D_g$ failing $(*)_{\delta}$ into $\mathcal W_{g+1}$. Fix such a sequence $\delta$ with associated elementary sequence $\varphi$. Choose the minimal integer $j \in Z$ such that $j < \min P$, the cycle $C$ containing $j$ is bad, and $t_C(j) = 1$. Such an integer exists by Lemma \ref{lem:boundary-vertex}.

Let $m = |C|$ denote the cardinality of the cycle. We order the elements of $C$ as follows 
\begin{equation}\label{eq:order_on_the_cycle}
    j, \ \pi_{\delta}^{-1}(j), \ \pi_{\delta}^{-2}(j), \ \ldots, \ \pi_{\delta}^{-(m-1)}(j).
\end{equation}
Let $z_1, \ldots, z_q$ denote the elements $z \in C \cap Z$ such that $t_C(z) = 1$, numbered as they appear in the order specified above. In particular $z_1 = j$. Recall the function $r$ defined in \eqref{eq:chain-position}. For every $1 \leq i \leq q$, we consider 
\begin{equation*}
    \ell_i \coloneq r(z_i) + 1.
\end{equation*}
Thus $\ell_i$ is the cardinality of the chain of $\pi_{\delta}$ ending at $z_i$. For every $i$, consider the integer $y_i \coloneq \pi_{\delta}^{-\ell_i}(z_i)$. Then $\pi_{\delta}(y_i) = \varphi_g + k$ for some $1 \leq k \leq g - \varphi_g$, thus we must have $y_i \in Z$. Moreover, the orientation $t_C$ is constant equal to $1$ along the chain containing $z_i$, so \eqref{eq:orientation} gives $t_C(y_i) = -1$. The integer $\pi_{\delta}^{-1}(y_i)$ can not belong to $P$, as otherwise we would have $t_C(\pi_{\delta}^{-1}(y_i)) = -1$, contradicting the fact that $C$ is bad. Therefore $\pi_{\delta}^{-1}(y_i) \in Z$, and the chain ending at $y_i$ consists of $y_i$ alone. Moreover $t_C(\pi_{\delta}^{-1}(y_i)) = 1$, so that we actually have $\pi_{\delta}^{-1}(y_i) = z_{i+1}$, where we put $z_{q+1} \coloneq z_1$. 

Consequently, if we write down the list of cardinalities of the successive chains of $\pi_{\delta}$ within $C$, starting at $j$ and proceeding in the ordered specified in \eqref{eq:order_on_the_cycle}, we obtain
\begin{equation}\label{eq:bad-run-representative}
    (\ell_1, 1, \ell_2, 1, \ldots, \ell_q, 1), \qquad m = \sum_{i=1}^q (\ell_i+1).
\end{equation}
If $\ell_i > \ell_1$, Lemma \ref{lem:run-length-order} would give $z_i < j < \min P$, contradicting the choice of $j$. Thus $\ell_i \leq \ell_1$ for every $i$.

Furthermore, every $i \in P$ satisfies $i>j$, hence $r(i) \leq r(j) = \ell_1-1$. For any chain $T_k$ of length $a_k \geq 2$, if $z \in Z$ denotes its last integer, we have $\pi_{\delta}^{-1}(z)\in P$ and
\begin{equation*}
    a_k - 2 = r(\pi_{\delta}^{-1}(z)) \leq \ell_1-1.
\end{equation*}
Therefore every part of the composition $c(\delta)$, as defined in \eqref{eq:chain-composition}, is at most $\ell_1+1$.

Let us now consider the complement $E \coloneq \{1, \ldots, g\} \setminus C$ of the bad cycle $C$, and write $n \coloneq |E|$. Suppose first that $n > 0$, and write $e_1 < \ldots < e_n$ for the elements of $E$ ordered increasingly. Since $E$ is stable by $\pi_{\delta}$, there is a unique permutation $w \in \mathfrak S_n$ such that
\begin{equation}\label{eq:identity_pi_w}
    \forall 1 \leq k \leq n, \qquad \pi_{\delta}(e_k) = e_{w(k)}.
\end{equation}
Define $\eta = (\eta_1, \ldots, \eta_n)$ by the formula $\eta_k \coloneq \delta_{e_k}$. If $\eta_1 = 1$, then $e_1 \in P$ and $\pi_{\delta}(e_1) < e_1$, contradicting the minimality of $e_1$ in $E$. Thus $\eta \in \mathcal D_n$.

Let $P_{\eta}$ and $Z_{\eta}$ be the subsets of $\{1, \ldots, n\}$ defined by $\eta$. Since the indices $e_k$ are increasing, the equation \eqref{eq:identity_pi_w} implies that $w$ is increasing on $P_{\eta}$, decreasing on $Z_{\eta}$, and satisfies $w(x) < w(y)$ for every $x \in P_{\eta}$ and $y \in Z_{\eta}$. Since $w$ is a bijection, it follows that
\begin{equation*}
    w(P_{\eta}) = \{1, \ldots, |P_{\eta}|\}, \qquad w(Z_{\eta}) = \{|P_{\eta}|+1, \ldots, n\}.
\end{equation*}
Together with the stated monotonicity, these properties characterize $\pi_{\eta}$ by \eqref{eq:pi-definition}. Therefore we have $w = \pi_{\eta}$, and thus $\pi_{\delta}(e_k) = e_{\pi_{\eta}(k)}$ for all $1 \leq k \leq n$. 

The chains of $\pi_{\eta}$ in $\{1, \ldots, n\}$, as defined in \eqref{eq:definition_chains}, correspond bijectively to the chains of $\pi_{\delta}$ contained in $E$. In particular, they have the same respective lengths. Therefore, if we write $R \coloneq c(\eta)$ as defined in \eqref{eq:chain-composition}, then the components of $R$ are at most $\ell_1 + 1$ and their sum is $n$.

If $n=0$, we take $R$ to be empty instead.

Finally, we associate to $\delta$ the composition
\begin{equation}\label{eq:upper-composition}
    (\ell_1+1, \ \ell_2+1, \ \ldots, \ \ell_q+1, \ 1, \ R).
\end{equation}
Its components sum to $m + 1 + n = g+1$. Since $\ell_i \leq \ell_1$ and every entry of $R$ is at most $\ell_1 + 1$, it belongs to $\mathcal W_{g+1}$.

It remains to prove injectivity of this construction. Suppose that $\delta, \delta' \in \mathcal D_g$ give the same composition \eqref{eq:upper-composition}. We denote every object introduced above relative to $\delta'$ with the same symbols primed, such as $C'$, $j'$, $E'$, etc. The first occurrence of $1$ in the composition \eqref{eq:upper-composition} completely determines $q$, $\ell_1, \ldots, \ell_q$ and $R$, so these data agree for both $\delta$ and $\delta'$. In particular,
\begin{equation*}
    |C| = |C'| = m, \qquad m = \sum_{i=1}^q (\ell_i+1),
\end{equation*}
and both complements $E$ and $E'$ have size $n = g-m$.

By the construction of the chains within $C$, the sequence $\left( \delta_{\pi_{\delta}^{-\nu}(j)} \right)_{0 \leq \nu < m}$ is the concatenation, for $i = 1, \ldots, q$, of the blocks
\begin{equation*}
    (0, \underbrace{1, \ldots, 1}_{\ell_i-1}, 0).
\end{equation*}
The same holds for $C'$, starting at $j'$. Consequently,
\begin{equation*}
    \forall 0 \leq \nu < m, \qquad \delta_{\pi_{\delta}^{-\nu}(j)} = \delta'_{\pi_{\delta'}^{-\nu}(j')}.
\end{equation*}
If $n > 0$, since both $\delta$ and $\delta'$ produce the same composition $R$ of $n$, by Proposition \ref{prop:composition-parametrization} we have
\begin{equation*}
    \eta = \delta(R) = \eta'.
\end{equation*}

Define a function $f: \{1, \ldots, g\} \to \{1, \ldots, g\}$ by the formula
\begin{align*}
    \forall 0 \leq \nu < m, & & f(\pi_{\delta}^{-\nu}(j)) & = \pi_{\delta'}^{-\nu}(j'),\\
    \forall 1 \leq k \leq n, & & f(e_k) &  = e'_k.
\end{align*}
The second formula is omitted when $n = 0$. These formulas define a bijection, sending $C$ onto $C'$ and $E$ onto $E'$. Moreover, we have checked that $\delta'_{f(i)} = \delta_i$ for every $1 \leq i \leq g$.

Lastly, the first formula defining $f$ actually holds for every integer $\nu$, since both cycles $C$ and $C'$ have length $m$. Thus
\begin{equation*}
    f \left( \pi_{\delta} (\pi_{\delta}^{-\nu}(j)) \right) = \pi_{\delta'}^{1-\nu}(j') = \pi_{\delta'} \left( f(\pi_{\delta}^{-\nu}(j)) \right).
\end{equation*}
On the complement $E$, the equation \eqref{eq:identity_pi_w} for the two sequences and the equality $\eta = \eta'$ give
\begin{equation*}
    f(\pi_{\delta}(e_k)) = f(e_{\pi_{\eta}(k)}) = e'_{\pi_{\eta}(k)} = \pi_{\delta'}(e'_k) =\pi_{\delta'}(f(e_k)).
\end{equation*}
Hence $f\circ \pi_{\delta} = \pi_{\delta'} \circ f$. By Lemma \ref{lem:cycle-rigidity}, we obtain $\delta = \delta'$, and this concludes the proof.
\end{proof}

\subsection{Some estimates for $N_g$}

For the numbers $W_n$, we have the following easy estimates.

\begin{lemma}\label{lem:Wn-asymptotic-order}
    For every $n \geq 1$, we have
    \begin{equation*}
        \frac{2^n-1}{n} \leq W_n \leq \frac{2^{n+1}-2}{n+1}.
    \end{equation*}
\end{lemma}

\begin{proof}
Define an equivalence relation $\sim$ on the set of compositions of $n$, by declaring that $c \sim c'$ if and only if $c$ can be obtained from $c'$ after a cyclic rotation of its components. More precisely, if $c = (c_1, \ldots, c_r)$ and $c' = (c'_1, \ldots, c'_{r'})$, then $c \sim c'$ if and only if $r = r'$ and there exists some $k$ such that $c_i = c'_{k+i}$ for all $1 \leq i \leq r$, where the indices are read modulo $r$. The equivalence class of a given $c = (c_1, \ldots, c_r)$ is of cardinality at most $r$, and it contains at least one element of $\mathcal W_n$. Since there are exactly $\binom{n-1}{r-1}$ compositions of $n$ with $r$ components, summing over $r$ gives
\begin{equation*}
    W_n \geq \sum_{r=1}^n \frac{1}{r}\binom{n-1}{r-1} = \frac{1}{n} \sum_{r=1}^n \binom{n}{r} = \frac{2^n-1}{n}.
\end{equation*}

For the upper bound, fix $c = (c_1, \ldots, c_r) \in \mathcal W_n$. Replace the first component $c_1$ by the sequence $1\ 0^{c_1}$, and all the remaining components $c_i$ by $1\ 0^{c_i-1}$ for $2 \leq i \leq r$, where $0^m$ means a sequence of $m$ successive $0$'s. This produces a non-constant sequence $\beta \in \{0, 1\}^{n+1}$. Since $c_1 > c_i-1$ for every $2 \leq i \leq r$, the equivalence class of $\beta$ under cyclic rotations of its components determines $\beta$ itself. Besides, the $n+1$ rotations of $\beta$ are all distinct sequences in $\{0, 1\}^{n+1}$, and the assignment mapping $c$ to the equivalence class of $\beta$ is injective on $\mathcal W_n$. We obtain $(n+1) W_n \leq 2^{n+1} -2$, which is the desired upper bound.
\end{proof}

\begin{proof}[Proof of Theorem \ref{thm:asymptotic-count}]
Proposition \ref{prop:failure-sandwich} and Lemma \ref{lem:Wn-asymptotic-order} give, for every $g \geq 3$,
\begin{equation*}
    \frac{2^{g-2}-1}{g-2} \leq f_g \leq \frac{2^{g+2}-2}{g+2}.
\end{equation*}
Dividing by $2^{g-1}$, we obtain
\begin{equation*}
    \frac{1 - 2^{2-g}}{2(g-2)} \leq 1 - \frac{N_g}{2^{g-1}} = \frac{f_g}{2^{g-1}} \leq \frac{8 - 2^{2-g}}{g+2}.
\end{equation*}
Letting $g$ go to infinity, this yields $N_g \sim 2^{g-1}$ as desired.
\end{proof}

\begin{remark}\label{rem:Knopfmacher-Robbins}
    The inequalities of Theorem \ref{thm:asymptotic-count} show that $f_g = \Theta(W_g)$ as $g$ goes to infinity. The author does not know whether the sharper estimate $f_g \sim W_g$ holds. We point out, however, that an asymptotic estimate of $W_n$ is known. Indeed, Theorem 2 of \cite{knopfmacher_compositions_2005} states that 
    \begin{equation*}
        W_n \sim \frac{2^n}{n\log(2)} \left( 1 + \varepsilon(\log_2 n) \right),
    \end{equation*}
    where $\varepsilon$ (denoted $\delta$ in \cite{knopfmacher_compositions_2005}) is an explicit continuous, non-constant periodic function of period $1$, with mean zero and ``small amplitude''. In fact, the amplitude of $\varepsilon$ is less than $10^{-5}$. 
\end{remark}

\section{An application to unitary Ekedahl-Oort strata} \label{sec:unitary-application}

Let $k$ be an algebraically closed field of characteristic $p$, and fix an imaginary quadratic field $E$ in which $p$ is inert. For $a \geq b \geq0$ with $a+b > 0$, write $\mathcal M(a,b)$ for the base change to $k$ of the special fiber at $p$ of a hyperspecial integral model of a unitary PEL Shimura variety associated with $E$ and of signature $(a,b)$, as in Section 2.1 of \cite{andrewsEkedahlOortStrata$mathsfGUq22$2025a}. We denote its supersingular locus by $\mathcal M(a,b)^{\mathrm{ss}}$.

Let $T$ be a unitary $\mathrm{BT}_1$-module of signature $(a,b)$, as defined in \cite{andrews_classification_2026} Definition 4.1, and denote its EO stratum by $\mathcal M(a,b)_T$. Forgetting the action gives a polarized $\mathrm{BT}_1$-module isomorphic to $M_{\delta} \otimes_{\mathbb F_p} k$ for a unique $\delta \in \{0,1\}^{a+b}$. Theorem \ref{thm:main} gives
\begin{equation*}
    \mathcal M(a,b)_T \cap \mathcal M(a,b)^{\mathrm{ss}} \not = \emptyset \implies (*)_{\delta}.
\end{equation*}
Indeed, the underlying $p$-divisible group of a supersingular point is a principally quasi-polarized supersingular lift of $M_{\delta}$, as already argued in \cite{andrewsEkedahlOortStrata$mathsfGUq22$2025a} Section 5.1. 

\begin{corollary}\label{cor:unitary-arbitrary}
    For $a \geq b \geq 1$, at least $N_b$ distinct EO strata of $\mathcal M(a,b)$ meet its supersingular locus.
\end{corollary}

\begin{proof}
Assume first that $a > b$. Choose $C \in \mathcal M(a-b, 0)(k)$. Every such point is supersingular by \cite{andrewsEkedahlOortStrata$mathsfGUq22$2025a} Paragraph 4.3.1. For each $\delta \in \mathcal D_b$ satisfying $(*)_{\delta}$, choose
\begin{equation*}
    (A_{\delta}, \lambda_{\delta}) \in \mathcal A_{b,\delta}(k) \cap \mathcal S_b(k),
\end{equation*}
and let us consider $Y_{\delta} \coloneq (\mathcal O_E \otimes_{\mathbb Z} A_{\delta}) \times C$. The Serre tensor construction of \cite{andrews_classification_2026} Section 4.5, and the product maps of \cite{andrewsEkedahlOortStrata$mathsfGUq22$2025a} Section 4, equips $Y_{\delta}[p^{\infty}]$ with a unitary PEL structure of signature $(a, b)$. By Theorem 1.6 (2) of \cite{viehmann_ekedahloort_2013}, there is a point $B_{\delta} \in \mathcal M(a,b)(k)$ whose $p$-divisible group, with its additional structures, is isomorphic to $Y_{\delta}[p^{\infty}]$. Since the underlying abelian variety of $Y_{\delta}$ is isomorphic to $A_{\delta}^2\times C$, the point $B_{\delta}$ is supersingular. By \cite{andrews_classification_2026} Lemma 4.27, its underlying polarized $\mathrm{BT}_1$-module is
\begin{equation*}
    \mathbb D(B_{\delta}[p]) \simeq \left( M_{\delta} \otimes_{\mathbb F_p} k \right)^{\oplus 2} \oplus \mathbb D(C[p]).
\end{equation*}
By the uniqueness of the decomposition of polarized $\mathrm{BT}_1$-modules into a direct sum of indecomposable summands, as recalled in \cite{andrews_classification_2026} Proposition 3.11, any two distinct $\delta \not = \delta'$ produce two non-isomorphic polarized $\mathrm{BT}_1$-modules $\mathbb D(B_{\delta}[p]) \not \simeq \mathbb D(B_{\delta'}[p])$. Therefore, the $N_b$ choices of $\delta$ give distinct unitary EO strata, all meeting the supersingular locus. When $a = b$, the same argument applies directly to $Y_{\delta} = \mathcal O_E \otimes_{\mathbb Z} A_{\delta}$, without the factor $C$.
\end{proof}

\bibliographystyle{amsplain}
\bibliography{biblio}

\end{document}